\documentclass{amsart}

\usepackage[a4paper,hmargin=3.70cm,vmargin=4cm]{geometry}
\usepackage{amsfonts,amssymb,amscd,amstext}
\usepackage{graphicx}
\usepackage[dvips]{epsfig}
\usepackage{hyperref}

\renewcommand{\leq}{\leqslant}
\renewcommand{\geq}{\geqslant}
\newcommand{\ptl}{\partial}
\newcommand{\rr}{{\mathbb{R}}}
\newcommand{\rrn}{\mathbb{R}^{n+1}}
\newcommand{\sph}{\mathbb{S}}
\newcommand{\sub}{\subset}
\newcommand{\subeq}{\subseteq}
\newcommand{\escpr}[1]{\big<#1\big>}
\newcommand{\Sg}{\Sigma} \newcommand{\sg}{\sigma}
\newcommand{\Om}{\Omega}
\newcommand{\eps}{\varepsilon}
\newcommand{\ric}{\text{Ric}}
\newcommand{\ind}{\mathcal{Q}}
\newcommand{\indo}{\mathcal{I}}
\newcommand{\cone}{{\times\!\!\!\!\times}}

\DeclareMathOperator{\divv}{div}

\newtheorem{theorem}{Theorem}[section]
\newtheorem{proposition}[theorem]{Proposition}
\newtheorem{lemma}[theorem]{Lemma}
\newtheorem{corollary}[theorem]{Corollary}

\theoremstyle{definition}

\newtheorem{remark}[theorem]{Remark}
\newtheorem{example}[theorem]{Example}

\theoremstyle{remark}

\numberwithin{equation}{section}

\begin{document}

\bibliographystyle{amsplain}  

\title[Free boundary stability and rigidity in manifolds with density]
{Free boundary stable hypersurfaces in manifolds with density and rigidity results}

\author[K.~Castro]{Katherine Castro}
\address{Departamento de Geometr\'{\i}a y Topolog\'{\i}a \\
Universidad de Granada \\ E--18071 Granada \\ Spain}
\email{ktcastro@ugr.es}

\author[C.~Rosales]{C\'esar Rosales}
\address{Departamento de Geometr\'{\i}a y Topolog\'{\i}a \\
Universidad de Granada \\ E--18071 Granada \\ Spain}
\email{crosales@ugr.es}

\date{\today}

\thanks{The authors have been supported by MICINN-FEDER grant MTM2010-21206-C02-01, and Junta de Andaluc\'ia grants FQM-325 and P09-FQM-5088.} 



\begin{abstract}
Let $M$ be a weighted manifold with boundary $\ptl M$, i.e., a Riemannian manifold where a density function is used to weight the Riemannian Hausdorff measures. In this paper we compute the first and the second variational formulas of the interior weighted area for deformations by hypersurfaces with boundary in $\ptl M$. As a consequence, we obtain variational characterizations of critical points and second order minima of the weighted area with or without a volume constraint. Moreover, in the compact case, we obtain topological estimates and rigidity properties for free boundary stable and area-minimizing hypersurfaces under certain curvature and boundary assumptions on $M$. Our results and proofs extend previous ones for Riemannian manifolds (constant densities) and for hypersurfaces with empty boundary in weighted manifolds.
\end{abstract}

\maketitle

\thispagestyle{empty}

\section{Introduction}
\label{sec:intro}
\setcounter{equation}{0}

Stable hypersurfaces in a Riemannian manifold with boundary are second order minima of the \emph{interior area} for compactly supported deformations preserving the boundary of the manifold and,  possibly, the volume separated by the hypersurface. From the first variation formulas \cite{ros-vergasta} such hypersurfaces have constant mean curvature and free boundary meeting orthogonally the boundary of the manifold. Moreover, the second variation formula \cite{ros-vergasta} implies that the associated index form is nonnegative for functions with compact support (and mean zero if the volume-preserving condition is assumed). The stability property has been extensively discussed and plays a central role in relation to classical minimization problems such as the Plateau problem or the isoperimetric problem.

The study of variational questions associated to the area functional in \emph{manifolds with density}, also called \emph{weighted manifolds} or \emph{smooth mm-spaces}, has been focus of attention in the last years. A \emph{manifold with density} is a connected Riemannian manifold, possibly with boundary, where a smooth positive function is used to weight the Hausdorff measures associated to the Riemannian distance. This kind of structures have been considered by many authors and provide a generalization of Riemannian geometry which is currently of increasing interest. For a nice introduction to weighted manifolds we refer the reader to Chapter 18 of Morgan's book \cite{gmt} and to Chapter 3 of Bayle's thesis \cite{bayle-thesis}. In the present paper we study \emph{free boundary stable hypersurfaces} in manifolds with density, by obtaining variational characterizations, topological and geometrical information, and rigidity results for the ambient manifold. In order to describe our results in more detail we need to introduce some notation and definitions.

Let $M$ be a Riemannian manifold endowed with a density $f=e^\psi$. We consider a smooth oriented hypersurface $\Sg$ immersed in $M$ in such a way that $\text{int}(\Sg)\sub\text{int}(M)$ and $\ptl\Sg\sub\ptl M$ whenever $\ptl\Sg\neq\emptyset$. We say that $\Sg$ is \emph{strongly $f$-stationary} if it is a critical point of the \emph{weighted area functional} under compactly supported deformations preserving the boundary $\ptl M$. Note that the \emph{weighted area} $A_f(\Sg)$ defined in \eqref{eq:volarea} is relative to the interior of $M$, so that $\Sg\cap\ptl M$ does not contribute to $A_f(\Sg)$. If, in addition, the hypersurface $\Sg$ has non-negative second derivative of the weighted area for any variation, then we say that $\Sg$ is \emph{strongly $f$-stable}. In the Riemannian setting (constant density $f=1$), these definitions coincide with the classical notions of free boundary minimal and stable hypersurfaces. Recently, many authors have considered complete strongly $f$-stable hypersurfaces \emph{with empty boundary}, see \cite{fan}, \cite{ho}, \cite{calibrations}, \cite{mejia}, \cite{espinar} and \cite{liu}, among others. However, not much is known about strongly $f$-stable hypersurfaces \emph{with non-empty boundary}, and this has been in fact our main motivation in the present work.

Our first aim in this paper is to provide variational characterizations of strongly $f$-stationary and stable hypersurfaces in the same spirit of the ones given by Ros and Vergasta in \cite{ros-vergasta} for the Riemannian case. This is done in Section~\ref{sec:var}, where we follow the arguments for hypersurfaces with empty boundary in \cite[Ch.~3]{bayle-thesis} and \cite{rcbm}, in order to compute the first and second derivatives of the weighted area. As a consequence, we deduce that a hypersurface $\Sg$ in $M$ is strongly $f$-stationary if and only if it has vanishing $f$-mean curvature and meets $\ptl M$ orthogonally in the points of $\ptl\Sg$, see Corollary~\ref{cor:stationary}. The $f$-\emph{mean curvature} of $\Sg$ is the function $H_f$ in \eqref{eq:fmc} previously introduced by Gromov \cite{gromov-GAFA} in relation to the first derivative of the weighted area, see Lemma~\ref{lem:1st}. We also show that the strong $f$-stability of $\Sg$ is equivalent to that the associated $f$-index form defined in \eqref{eq:indo} is nonnegative for smooth functions with compact support, see Corollary~\ref{cor:stable}. At this point, it is worth mentioning that the techniques employed in this section allow also to characterize critical points and second order minima of the weighted area for deformations \emph{preserving the weighted volume} $V_f$ defined in \eqref{eq:volarea}. This is closely related to the \emph{partitioning problem}, which consists of separating a given weighted volume in $M$ with the least possible interior weighted area. However, besides showing some relevant situations in Examples~\ref{ex:1},~\ref{ex:2} and ~\ref{ex:3}, the partitioning problem and the associated $f$-stable hypersurfaces will not be treated in detail. Some characterization results for compact $f$-stable hypersurfaces with free boundary in a Euclidean solid cone where a homogeneous density is considered can be found in \cite{homostable}. 

In the remainder of the paper we mainly investigate the relationship between the topology of compact strongly $f$-stable hypersurfaces and the geometry of the ambient manifold by means of the second variation formula. As a motivation, note that the $f$-index form in \eqref{eq:indo} of a hypersurface $\Sg$ is a quadratic form which involves the extrinsic geometry of $\Sg$, the second fundamental form $\text{II}$ of $\ptl M$, and the \emph{Bakry-\'Emery-Ricci curvature} $\ric_f$ of $M$ defined in \eqref{eq:fricci}. The $2$-tensor $\ric_f$ was first introduced by Lichnerowicz \cite{lich1}, \cite{lich2}, and later generalized by Bakry and \'Emery \cite{be} in the framework of diffusion generators. In particular, it is easy to observe that the stability inequality becomes more restrictive provided $\text{II}$ and $\ric_f$ are always semidefinite positive. Hence \emph{local convexity} of $\ptl M$ and \emph{nonnegativity of the Bakry-\'Emery-Ricci curvature} become natural hypotheses in order to obtain interesting consequences from the stability condition.

In Section~\ref{subsec:ricci} we establish some results in this direction. In fact, by assuming $\ric_f\geq 0$ and $\text{II}\geq 0$ we deduce in a quite straightforward way that a compact strongly $f$-stable hypersurface must be totally geodesic, see Lemma~\ref{lem:easy}. Moreover, we also have $\ric_f(N,N)=0$ and $\text{II}(N,N)=0$, where $N$ is the unit normal to $\Sg$. In particular, if $\ric_f>0$ or $\text{II}>0$, then there are no compact strongly $f$-stable hypersurfaces in $M$. This property was observed by Simons \cite{simons-james} for the Riemannian case, and later generalized by Fan \cite{fan}, and Cheng, Mejia and Zhou \cite{mejia} for hypersurfaces with empty boundary in manifolds with density. On the other hand, Espinar showed in \cite{espinar} that Lemma~\ref{lem:easy} also holds for complete strongly $f$-stable hypersurfaces of finite type and empty boundary. 

The simplest examples of strongly $f$-stable totally geodesic hypersurfaces satisfying $\ric_f(N,N)=0$ and $\text{II}(N,N)=0$ are the horizontal slices $\{s\}\times\Sg$ in a Riemannian product $\rr\times\Sg$, where $\Sg$ is a compact Riemannian manifold of non-negative Ricci curvature, and the logarithm of the density $f$ is a linear function in $\rr$. These are not the unique examples we may give, i.e., the existence of compact strongly $f$-stable hypersurfaces in the above conditions does not imply that the metric of $M$ splits, even locally, as a product metric, see \cite{micallef}. However, in Theorem~\ref{th:ricci} we prove the following rigidity result: 
\begin{quotation}
\emph{If a weighted manifold $M$ with non-negative Bakry-\'Emery-Ricci curvature and locally convex boundary contains a compact, oriented, embedded, locally weighted area-minimizing hypersurface $\Sg$ with non-empty boundary, then there is a neighborhood of $\Sg$ in $M$ isometric to a Riemannian product $(-\eps_0,\eps_0)\times\Sg$}. 
\end{quotation}
This local result can be globalized by means of a standard continuation argument. As a consequence, if we further assume that $M$ is complete and $\Sg$ \emph{minimizes the weighted area in its isotopy class}, then $M$ is a Riemannian quotient of $\rr\times\Sg$. We must remark that this rigidity result was previously obtained by Liu \cite{liu} for weighted area-minimizing hypersurfaces \emph{with empty boundary}. To prove it, Liu used the second variation formula to analyze the weighted area functional for the deformation by normal geodesics leaving from $\Sg$. In our context, however, a normal geodesic starting from $\ptl\Sg$ is not necessarily confined to stay in $M$, and so this deformation cannot be considered. As we will explain in more detail later, this difficulty is solved by taking another deformation which moves $\ptl\Sg$ along $\ptl M$.

In Section~\ref{subsec:scalar} we provide a topological restriction for strongly $f$-stable surfaces, and a rigidity result for weighted area-minimizing surfaces in a weighted $3$-manifold $M$ of \emph{non-negative Perelman scalar curvature} and \emph{$f$-mean convex boundary}.  On the one hand, the \emph{Perelman scalar curvature} $S_f$ defined in \eqref{eq:fscalar} is the generalization of the Riemannian scalar curvature introduced by Perelman \cite{perelman} when showing that the Ricci flow is a gradient flow. Let us indicate, as a remarkable difference with respect to the Riemannian case, that the Perelman scalar curvature \emph{is not the trace} of the Bakry-\'Emery-Ricci curvature. In fact, we have the Bianchi identity $S_f=2\,\nabla^*\ric_f$, where $\nabla^*$ is the adjoint operator of $\nabla$ with respect to the $L^2$-norm for the weighted volume measure $dv_f:=f\,dv$. We also remark that $S_f$ is the limit as $n$ tends to infinity of the conformally invariant scalar curvature $S_f^n$ introduced by Chang, Gursky and Yang, see \cite{chang} for a rigorous statement. On the other hand, the \emph{$f$-mean convexity} of $\ptl M$ means that the $f$-mean curvature of $\ptl M$ is nonnegative when computed with respect to the inner unit normal. 

There are several works on the topology of compact stable minimal surfaces in $3$-manifolds of non-negative scalar curvature and mean convex boundary. Schoen and Yau proved in \cite{schoen-yau} that, if such a surface $\Sg$ is immersed in a Riemannian $3$-manifold $M$ of positive scalar curvature, then it must be topologically a sphere. This was later generalized by Fischer-Colbrie and Schoen \cite{fcs}, who showed that, if $M$ has non-negative scalar curvature, then $\Sg$ is a sphere or a totally geodesic flat torus. These results have been extended for surfaces with empty boundary in manifolds with density by Fan~\cite{fan} and Espinar~\cite{espinar}, respectively. For the case of non-empty boundary, Chen, Fraser and Pang~\cite{fbr}, and Ambrozio~\cite{ambrozio}, have recently proved that a compact free boundary stable minimal surface inside a $3$-manifold of non-negative scalar curvature and mean convex boundary is either a disk or a totally geodesic flat cylinder. This was previously established in $\rr^3$ by Ros \cite{ros-free}, who also showed that stable cylinders cannot appear. In Theorem~\ref{th:main} we obtain the following:
\begin{quotation}
\emph{A smooth, compact, oriented, strongly $f$-stable surface $\Sg$ with non-empty boundary inside a $3$-manifold $M$ of non-negative Perelman scalar curvature and $f$-mean convex boundary is either a disk or a totally geodesic flat cylinder bounded by geodesics in $M$.}
\end{quotation} 
As in the previous results, our proof is based on the second variation for the area, the Gauss formula and the Gauss-Bonnet theorem. Moreover, after a careful analysis we deduce that along a strongly $f$-stable cylinder we have $\ric_f(N,N)=0$, $\text{II}(N,N)=0$ and the density $f$ must be constant.

The existence of strongly $f$-stable cylinders in the previous conditions cannot be discarded. In fact, in the Riemannian product $\rr\times\Sg$ with constant density, where $\Sg$ is the compact cylinder $\sph^1\times [a,b]$ endowed with the Euclidean metric of $\rr^3$, any horizontal slice $\{s\}\times\Sg$ provides a strongly $f$-stable cylinder. As in the case $\ric_f\geq 0$ and $\ptl M$ locally convex, other examples may be given where the Riemannian metric of $M$ does not split as a product metric. However, in Theorem~\ref{th:main2} we prove the following rigidity result: 
\begin{quotation}
\emph{If a weighted $3$-manifold $M$ of non-negative Perelman scalar curvature and $f$-mean convex boundary contains an oriented, embedded, locally weighted area-minimizing cylinder $\Sg$, then there is a neighborhood of $\Sg$ in $M$ which is isometric to $(-\eps_0,\eps_0)\times\Sg$, and the density $f$ is constant in such neighborhood. Moreover, if $M$ is complete and $\Sg$ minimizes the weighted area in its isotopy class, then $\rr\times\Sg$ is an isometric covering of $M$, and the density $f$ is constant on $M$.} 
\end{quotation}
We remark that rigidity results for area-minimizing tori and cylinders in Riemannian $3$-manifolds of non-negative scalar curvature and mean convex boundary were previously established by Cai and Galloway~\cite{cai-galloway}, and by Ambrozio~\cite{ambrozio}. An extension of Cai and Galloway's result for surfaces with empty boundary in manifolds with density has been obtained by Espinar \cite{espinar}. At this point, it is worth mentioning that our rigidity results in Theorems~\ref{th:ricci} and \ref{th:main2} are independent since the hypotheses $\ric_f\geq 0$ and $\ptl M$ locally convex do not necessarily imply $S_f\geq 0$ and $f$-mean convexity of $\ptl M$.

Finally, in Section~\ref{subsec:area} we show how the proofs of Theorems~\ref{th:main} and \ref{th:main2} can be adapted to provide optimal upper and lower bounds for the weighted area of a compact strongly $f$-stable surface in terms of a lower non-vanishing bound for the Perelman scalar curvature of a weighted $3$-manifold with $f$-mean convex boundary. The sharpness of these area estimates comes from the fact that, in case of equality for a locally weighted area-minimizing surface, we get the corresponding rigidity results, see Theorems~\ref{th:main3} and \ref{th:main4} for detailed statements. Previous area estimates in Riemannian $3$-manifolds were given by Shen and Zhu~\cite{shen-zhu} when $\ptl\Sg=\emptyset$, and by Chen, Fraser and Pang~\cite{fbr} when $\ptl\Sg\neq\emptyset$. The associated rigidity results were obtained by Bray, Brendle and Neves \cite{bbn} under a positive lower bound on the scalar curvature, and by Nunes~\cite{nunes} under a negative one. Extensions of these results for surfaces with empty boundary in manifolds with density were found by Espinar~\cite{espinar}.
A unified approach for surfaces with non-empty boundary in Riemannian $3$-manifolds has been given by Ambrozio~\cite{ambrozio}. 

We finish this introduction by explaining the geometric approach employed to prove our local rigidity results. Starting from a strongly $f$-stationary, totally geodesic hypersurface $\Sg$, with $\ric_f(N,N)=0$ and $\text{II}(N,N)=0$, we use the implicit function theorem to find an open neighborhood $\Om$ of $\Sg$ in $M$ which is foliated by a one-parameter family of strongly $f$-stationary hypersurfaces $\Sg_s$ with $\Sg_0=\Sg$, $\text{int}(\Sg_s)\sub\text{int}(M)$ and $\ptl\Sg_s\sub\ptl M$, see Proposition~\ref{prop:deformation}. Then, our curvature and convexity assumptions imply that, for such a family, the derivative of the $f$-mean curvature function is nonnegative. As a consequence, the associated weighted area is strictly decreasing unless any hypersurface $\Sg_s$ is totally geodesic and the normal component of the velocity vector along $\Sg_s$ is constant. Hence, if we start from a locally weighted area-minimizing hypersurface $\Sg$, then we can deduce that the normal vector field along the hypersurfaces $\Sg_s$ is parallel on the open set $\Om$. From here, it is not difficult to conclude that, for some $\eps_0>0$, the restriction to $(-\eps_0,\eps_0)\times\Sg$ of the normal exponential map associated to $\Sg$ is an isometry.

\section{Preliminaries}
\label{sec:preliminaries}
\setcounter{equation}{0}

Let $M$ be a smooth $(C^\infty)$ connected and oriented $(n+1)$-dimensional manifold with a Riemannian metric $g=\escpr{\cdot\,,\cdot}$. We denote by $\text{int}(M)$ and $\ptl M$ the interior and the boundary of $M$, respectively. By a \emph{density} on $M$ we mean a smooth positive function $f=e^\psi:M\to\rr$ used to weight the Hausdorff measures associated to the Riemannian distance. In particular, the \emph{weighted volume} of a Borel set $\Om\subseteq M$ and the \emph{$($interior$)$ weighted area} of a smooth hypersurface $\Sg$ are defined by
\begin{equation}
\label{eq:volarea}
V_f(\Om):=\int_\Om \,dv_f=\int_\Om f\,dv,\qquad 
A_f(\Sg):=\int_{\Sg\cap\text{int}(M)} da_f=\int_{\Sg\cap\text{int}(M)} f\,da, 
\end{equation}
where $dv$ and $da$ stand for the Riemannian elements of volume and area, respectively. According to the previous definition the set $\Sg\cap\ptl M$ does not contribute to $A_f(\Sg)$. We will also denote $dl_f:=f\,dl$, where $dl$ is the $(n-1)$-dimensional Hausdorff measure in $M$.

For a Riemannian manifold $(M,g)$ with density $f=e^\psi$, the \emph{Bakry-\'Emery-Ricci tensor} is defined by
\begin{equation}
\label{eq:fricci}
\text{Ric}_f:=\ric-\nabla^2\psi,
\end{equation}
where $\text{Ric}$ and $\nabla^2$ are the Ricci tensor and the Hessian operator in $(M,g)$. For us the Ricci tensor is given by $\ric(u,v):=\text{trace}(w\mapsto R(u,w)v)$, where $R$ is the curvature tensor in $(M,g)$. The \emph{Perelman scalar curvature} is the function
\begin{equation}
\label{eq:fscalar}
S_f:=S-2\,\Delta\psi-|\nabla\psi|^2,
\end{equation}
where $\Delta$ and $\nabla$ denote the Laplacian and the gradient in $(M,g)$, and $S$ is the scalar curvature given by $S(p):=\sum_{i=1}^{n+1}\ric_p(e_i,e_i)$, for any orthonormal basis $\{e_i\}$ of $T_pM$. For a constant density $f$, we have $\ric_f=\ric$ and $S_f=S$. Note also that $S_f$ \emph{does not coincide} in general with the trace of $\ric_f$.

Let $\Sg$ be a smooth oriented hypersurface immersed in $M$. For any smooth vector field $X$ along $\Sg$, we define the \emph{f-divergence relative} to $\Sg$ of $X$ by
\[
\divv_{\Sg,f}X:=\divv_\Sg X+\escpr{\nabla\psi,X},
\]
where $\divv_\Sg$ is the divergence relative to $\Sg$ in $(M,g)$. If $N$ is a unit normal vector along $\Sg$, then the \emph{f-mean curvature} of $\Sg$ with respect to $N$ is the function
\begin{equation}
\label{eq:fmc}
H_f:=-\divv_{\Sg,f}N=nH-\escpr{\nabla\psi,N},
\end{equation}
where $H:=(-1/n)\divv_\Sg N$ is the mean curvature of $\Sg$ in $(M,g)$. By using the Riemannian divergence theorem it was proved in \cite[Lem.~2.2]{homostable} that equality
\begin{equation}
\label{eq:divthsup}
\int_\Sg\divv_{\Sg,f}X\,da_f=-\int_\Sg H_f\,\escpr{X,N}\,da_f
-\int_{\ptl\Sg}\escpr{X,\nu}\,dl_f,
\end{equation}
holds for any smooth vector field $X$ with compact support on $\Sg$. Here we denote by $\nu$ the conormal vector, i.e., the inner unit normal to $\ptl\Sg$ in $\Sg$.

Finally, we define the \emph{f-Laplacian relative} to $\Sg$ of a function $u\in C^\infty(\Sg)$ by
\begin{equation}
\label{eq:deltaf}
\Delta_{\Sg,f}u:=\divv_{\Sg,f}(\nabla_\Sg u)=\Delta_\Sg u
+\escpr{\nabla_\Sg\psi,\nabla_\Sg u},
\end{equation} 
where $\nabla_\Sg$ is the gradient relative to $\Sg$. For this operator we have the following integration by parts formula, which is an immediate consequence of  \eqref{eq:divthsup}
\begin{equation}
\label{eq:ibp}
\int_\Sg u_1\,\Delta_{\Sg,f}\,u_2\,da_f=-\int_\Sg\escpr{\nabla_\Sg u_1,\nabla_\Sg u_2}\,da_f-\int_{\ptl\Sg}u_1\,\frac{\ptl u_2}{\ptl\nu}\,dl_f,
\end{equation}
where $u_1,u_2\in C^\infty_0(\Sg)$ and $\ptl u_2/\ptl\nu$ is the directional derivative of $u_2$ with respect to $\nu$.

\section{Stationary and stable free boundary hypersurfaces}
\label{sec:var}
\setcounter{equation}{0}

In this section we compute the first and the second derivative of area and volume for a variation of a hypersurface immersed in a manifold with density, and whose boundary lies in the boundary of the manifold. As a consequence, we characterize stationary points and second order minima of the area with or without a volume constraint, thus extending previous results in \cite{ros-vergasta} for Riemannian manifolds (constant density $f=1$), and in \cite[Ch.~3]{bayle-thesis} for hypersurfaces with empty boundary in weighted manifolds. We will follow closely the arguments in \cite[Sect.~3]{rcbm}, where hypersurfaces embedded in $\rrn$ and variations supported away from the boundary were considered. Note also that the first variational formulas for piecewise regular densities were established in \cite[Prop.~2.11]{cvm}.

Let $M$ be a smooth oriented Riemannian manifold endowed with a density $f=e^\psi$. We consider a smooth oriented hypersurface given by an immersion $\varphi_0:\Sg\to M$ such that $\varphi_0(\text{int}(\Sg))\sub\text{int}(M)$ and $\varphi_0(\ptl\Sg)\sub\ptl M$. If $\ptl\Sg=\emptyset$ then we adopt the convention that all the integrals along $\ptl\Sg$ vanish. We denote by $N$ the unit normal along $\Sg$ which is compatible with the orientations of $\Sg$ and $M$.

By a \emph{variation} of $\Sg$ we mean a smooth map $\varphi:(-\eps,\eps)\times\Sg\to M$ satisfying:
\begin{enumerate}
\item[(i)] for any $s\in (-\eps,\eps)$, the map $\varphi_s:\Sg\to M$ defined by $\varphi_s(p):=\varphi(s,p)$ is an immersion with $\varphi_s(\text{int}(\Sg))\sub\text{int}(M)$ and $\varphi_s(\ptl\Sg)\sub\ptl M$, 
\item[(ii)] $\varphi(0,p)=\varphi_0(p)$, for any $p\in\Sg$,
\item[(iii)] there is a compact set $C\subseteq\Sg$ such that $\varphi_s(p)=\varphi_0(p)$ for any $p\in\Sg-C$.
\end{enumerate}
The \emph{velocity vector} is the vector field $X_p:=(\ptl\varphi/\ptl s)(0,p)$ for any $p\in\Sg$. Note that $X$ has compact support and it is tangent to $\ptl M$ in the points of $\ptl\Sg$ by the condition (i) above. The function $A_f(s)$,  that maps any $s\in (-\eps,\eps)$ to the weighted area of $\Sg_s:=\varphi_s(\Sg)$ defined in \eqref{eq:volarea}, is the \emph{weighted area functional} associated to the variation. More explicitly
\begin{equation}
\label{eq:area}
A_f(s)=\int_\Sg (f\circ\varphi_s)\,|\text{Jac}\,\varphi_s|\,da.
\end{equation}
If $p\in\Sg$ and $\{e_i\}$ is any orthonormal basis in $T_p\Sg$, then $|\text{Jac}\,\varphi_s|(p)$ is the squared root of the determinant of the matrix $a_{ij}$ with $a_{ij}=\escpr{e_i(\varphi_s),e_j(\varphi_s)}$. We define the \emph{volume functional} $V_f(s)$ as in \cite[Sect.~2]{bdce}, i.e., $V_f(s)$ denotes the \emph{signed weighted volume} enclosed between $\Sg$ and $\Sg_s$. More precisely
\begin{equation}
\label{eq:volume}
V_f(s)=\int_{[0,s]\times C}\varphi^*(dv_f)=\int_{[0,s]\times C}(f\circ\varphi)\,\varphi^*(dv),
\end{equation}
where $dv_f=f\,dv$ is the weighted volume element in $M$. We say that the variation is \emph{volume preserving} if $V_f(s)$ is constant for any $s$ small enough. 

\begin{remark}
When $\Sg$ is an embedded hypersurface separating an open set $\Om\sub M$ with $V_f(\Om)<+\infty$, then we can associate to any variation of $\Sg$ a family of open sets $\Om_s\sub M$ such that $\Om_0=\Om$ and $\overline{\ptl\Om_s\cap\text{int}(M)}=\Sg_s$ for any $s\in I$. In this situation it is natural to define the volume functional by $\text{Vol}_f(s):=V_f(\Om_s)$. Observe that this functional does not coincide with the signed volume in \eqref{eq:volume}, which vanishes for $s=0$. However, we have $|V'_f(s)|=|\text{Vol}_f'(s)|$ for any $s\in (-\eps,\eps)$.
\end{remark}

In the next result we provide explicit expressions for the first derivatives of the functionals $A_f(s)$ and $V_f(s)$.

\begin{lemma}[First variation formulas]
\label{lem:1st} 
Let $M$ be a smooth oriented Riemannian manifold endowed with a density $f=e^\psi$. Consider a smooth oriented hypersurface $\Sg$ immersed in $M$ with $\emph{int}(\Sg)\sub\emph{int}(M)$ and $\ptl\Sg\sub\ptl M$. Given a variation $\varphi:(-\eps,\eps)\times\Sg\to M$ of $\Sg$ with velocity vector $X$, we have
\[
A_f'(0)=-\int_\Sg H_f\,u\, da_f-\int_{\ptl\Sg}\escpr{X,\nu}\,dl_f, \qquad V_f'(0)=\int_\Sg\,u\, da_f,
\]
where $H_f$ is the $f$-mean curvature of $\Sg$ defined in \eqref{eq:fmc}, $u$ is the normal component of $X$, and $\nu$ is the inner unit normal to $\ptl\Sg$ in $\Sg$.
\end{lemma}

\begin{proof}
By differentiating under the integral sign in \eqref{eq:area}, and taking into account that $(d/ds)|_{s=0}\,|\text{Jac}\,\varphi_s|=\divv_\Sg X$, see \cite[Sect.~9]{simon} and \cite[Lem.~5.4]{rosales-sr}, we get
\begin{align*}
A_f'(0)&=\int_{\Sg}\big(\escpr{\nabla f,X}+f\divv_\Sg X\big)\,da=\int_\Sg\divv_{\Sg,f}X\,da_f
\\
&=-\int_\Sg H_f\,u\, da_f-\int_{\ptl\Sg}\escpr{X,\nu}\,dl_f,
\end{align*}
where in the last equality we have used formula \eqref{eq:divthsup}.

Now we compute $V'_f(0)$. As in the proof of \cite[Lem.~(2.1)]{bdce} it is easy to see that
\[
\varphi^*(dv)(s,p)=\escpr{\frac{\ptl\varphi}{\ptl s}(s,p),N_s(p)}\,ds\wedge da,
\]
where $N_s$ is the unit normal along the immersion $\varphi_s:\Sg\to M$ which is compatible with the orientations of $\Sg$ and $M$. By using the definition of $V_f(s)$ in \eqref{eq:volume} and Fubini's theorem, we obtain
\[
V_f'(0)=\int_\Sg\escpr{X,N}\,(f\circ\varphi)\,da=\int_\Sg u\,da_f,
\]
which finishes the proof.
\end{proof}

We say that the hypersurface $\Sg$ is \emph{strongly $f$-stationary} if $A_f'(0)=0$ for any variation of $\Sg$. If $A_f'(0)=0$ for any volume-preserving variation of $\Sg$ then we will say that $\Sg$ is \emph{$f$-stationary}. From the expressions for $A_f'(0)$ and $V_f'(0)$ in Lemma~\ref{lem:1st} we can deduce the following characterization of stationary hypersurfaces.

\begin{corollary}
\label{cor:stationary}
Let $M$ be a smooth oriented Riemannian manifold endowed with a density $f=e^\psi$. Then, for a smooth oriented hypersurface $\Sg$ immersed in $M$ with $\emph{int}(\Sg)\sub\emph{int}(M)$ and $\ptl\Sg\sub\ptl M$, the following statements are equivalent
\begin{itemize}
\item[(i)] $\Sg$ is $f$-stationary $($resp. strongly $f$-stationary$)$.
\item[(ii)] The $f$-mean curvature of $\Sg$ defined in \eqref{eq:fmc} is a constant $H_0$ $($resp. vanishes$)$ and $\Sg$ meets $\ptl M$ orthogonally in the points of $\ptl\Sg$.
\item[(iii)] There is a constant $H_0$ such that $(A_f+H_0\,V_f)'(0)=0$ for any variation of $\Sg$ $($resp. $A_f'(0)=0$ for any variation of $\Sg$$)$.
\end{itemize} 
\end{corollary}

\begin{proof}
We give a proof when $\Sg$ is $f$-stationary (the case strongly $f$-stationary is easier). From Lemma~\ref{lem:1st} we can check that (ii) implies (iii), and that (iii) implies (i). To see that (i) implies (ii) we take a function $u\in C_0^\infty(\Sg)$ with $\text{supp}(u)\sub\text{int}(\Sg)$ and $\int_\Sg u\,da_f=0$. As in \cite[Lem~(2.2)]{bdce}, we find a volume-preserving variation of $\Sg$ whose velocity vector $X$ satisfies $\escpr{X,N}=u$ on $\Sg$. The fact that $\Sg$ is $f$-stationary yields $\int_\Sg H_fu\,da_f=0$ for any $u$ in the previous conditions, and so $H_f$ is constant along $\Sg$. Finally, suppose $\escpr{N_p,\xi_p}\neq 0$ for some $p\in\ptl\Sg$, where $\xi$ denotes the inner unit normal along $\ptl M$. Then, there exists a smooth vector field $Y$ with compact support on $\Sg$ such that $\int_\Sg\escpr{Y,N}\,da_f=0$ and $\escpr{Y,\nu}$ is a cut-off function along $\ptl\Sg$. By using again \cite[Lem.~(2.2)]{bdce} we could construct a volume-preserving variation of $\Sg$ with velocity vector $X$ satisfying $\escpr{X,N}=\escpr{Y,N}$ and $\escpr{X,\nu}=\escpr{Y,\nu}$. Hence we would get $0=A_f'(0)=-\int_{\ptl\Sg}\escpr{X,\nu}\,dl_f$, a contradiction.
\end{proof}

\begin{example}
\label{ex:1}
Let $M=0\cone\mathcal{D}$ be a cone over a smooth region $\mathcal{D}$ of the unit sphere $\sph^n$ in $\rrn$. If $f=e^\psi$ is a radial density, i.e., $\psi(p)$ only depends on $|p|$, then any sphere centered at the origin and intersected with $M$ is $f$-stationary since it has constant $f$-mean curvature, see \cite[Ex.~3.4]{rcbm}, and meets $\ptl M$ orthogonally. As was shown in \cite[Ex.~4.3]{homostable} this also holds if $f=e^\psi$ is a $k$-homogeneous density, i.e., $f(tp)=t^k\,f(p)$ for any $t>0$ and any $p\in M-\{0\}$. Different examples of $f$-stationary curves in planar sectors with density $f(p)=|p|^k$, $k>0$, were found in \cite{dhht}. On the other hand, if $M$ is a half-space or a slab in $\rrn$ with radial density, then the intersection with $M$ of any hyperplane perpendicular to $\ptl M$ and containing $0$ has vanishing $f$-mean curvature, and so it is strongly $f$-stationary. Moreover, for the Gaussian density $f(p)=e^{-|p|^2}$ any hyperplane perpendicular to $\ptl M$ is an $f$-stationary hypersurface. 
\end{example}

Next, we compute the second derivative of the functional $A_f+H_f\,V_f$ for an $f$-stationary hypersurface of constant $f$-mean curvature $H_f$.

\begin{proposition}[Second variation formula]
\label{prop:2nd}
Let $M$ be a smooth oriented Riemannian manifold endowed with a density $f=e^\psi$. Let $\varphi:(-\eps,\eps)\times\Sg\to M$ be a variation of a smooth oriented hypersurface $\Sg$ immersed in $M$ with $\emph{int}(\Sg)\sub\emph{int}(M)$ and $\ptl\Sg\sub\ptl M$. If $\Sg$ is $f$-stationary with constant $f$-mean curvature $H_f$, then we have
\[
(A_f+H_f\,V_f)''(0)=\indo_f(u,u),
\]
where $u$ is the normal component of the velocity vector and $\indo_f$ is the symmetric bilinear form on $C^\infty_0(\Sg)$ defined by
\begin{align}
\label{eq:indo}
\indo_f(v,w)&:=\int_\Sg\left\{\escpr{\nabla_\Sg v,\nabla_\Sg w}
-\big(\emph{Ric}_f(N,N)+|\sigma|^{2}\big)\,vw\right\}da_{f}
\\
\nonumber
&-\int_{\ptl\Sg}\emph{II}(N,N)\,vw\,dl_f.
\end{align}
In the previous expression $\emph{Ric}_f$ denotes the Bakry-\'Emery-Ricci tensor defined in \eqref{eq:fricci},  $\sg$ is the second fundamental form of $\Sg$ with respect to the unit normal $N$, and $\emph{II}$ is the second fundamental form of $\ptl M$ with respect to the inner unit normal.
\end{proposition}

\begin{proof}
First observe that the conormal vector $\nu$ to $\ptl\Sg$ coincides with the inner unit normal $\xi$ to $\ptl M$ along $\ptl\Sg$ by the orthogonality condition in Corollary~\ref{cor:stationary} (ii). Let us denote by $N_s$ the unit normal along $\Sg_s:=\varphi_s(\Sg)$ which is compatible with the orientations of $\Sg$ and $M$. By using Lemma~\ref{lem:1st} we obtain
\[
(A_f+H_f\,V_f)'(s)=-\int_{\Sg_s}(H_f)_s\,u_s\,(da_f)_s+
H_f\int_{\Sg_s}u_s\,(da_f)_s-\int_{\ptl\Sg_s}\escpr{X_s,\nu_s}\,(dl_f)_s,
\]
where $(H_f)_s$ is the $f$-mean curvature of $\Sg_s$, $(X_s)_p:=(\ptl\varphi/\ptl s)(s,p)$, $u_s:=\escpr{X_s,N_s}$ and $\nu_s$ is the conormal vector to $\ptl\Sg_s$. By differentiating into the previous equality and using that $\escpr{X,\nu}=0$ along $\ptl\Sg$, we get
\begin{equation}
\label{eq:2nd1}
(A_f+H_f\,V_f)''(0)=-\int_\Sg H_f'(0)\,u\,da_f-\int_{\ptl\Sg}\escpr{X_s,\nu_s}'(0)\,dl_f,
\end{equation}
where the primes in $H_f'(0)$ and $\escpr{X_s,\nu_s}'(0)$ denote differentiation along the curve $s\mapsto\varphi_s(p)$. On the one hand, the derivative $H_f'(0)$ was computed in the proof of \cite[Prop.~3.6]{rcbm} for the case of normal variations, see also \cite[Re.~3.7]{rcbm}. By taking into account that the $f$-mean curvature of $\Sg$ is constant, we have
\begin{equation}
\label{eq:hfprima}
H_f'(0)=\mathcal{L}_f(u):=\Delta_{\Sg,f}\,u+(\ric_f(N,N)+|\sg|^2)\,u,
\end{equation}
where $\Delta_{\Sg,f}$ is the $f$-Laplacian relative to $\Sg$ defined in \eqref{eq:deltaf}. On the other hand, the derivative $\escpr{X_s,\nu_s}'(0)$ can be computed as in \cite[Lem.~4.1~(2)]{ros-souam}, so that we obtain
\begin{equation}
\label{eq:bd}
\escpr{X_s,\nu_s}'(0)=u\,\left\{\frac{\ptl u}{\ptl\nu}+\text{II}(N,N)\,u\right\},
\end{equation}
where $\ptl u/\ptl\nu$ is the derivative of $u$ with respect to $\nu$. For further reference, it is worth noting that
\begin{equation}
\label{eq:mimo}
\escpr{\xi,N_s}'(0)=-\frac{\ptl u}{\ptl\nu}-\text{II}(X,N)-\sg(X^\top,\nu),
\end{equation}
where $X^\top$ is the tangent projection of $X$. This is an immediate consequence of equality
\begin{equation}
\label{eq:normal}
N_s'(0)=D_{X^\top}N-\nabla_\Sg u,
\end{equation}
which is proved in \cite[Lem.~4.1~(1)]{ros-souam}. By using \eqref{eq:hfprima} and \eqref{eq:bd}, equation \eqref{eq:2nd1} reads
\[
(A_f+H_f\,V_f)''(0)=\ind_f(u,u),
\]
where we define
\begin{equation*}
\ind_f(v,w):=-\int_\Sg v\,\mathcal{L}_f(w)\,da_f
-\int_{\ptl\Sg}v\,\left\{\frac{\ptl w}{\ptl\nu}+\text{II}(N,N)\,w\right\}dl_f.
\end{equation*}
Finally, an application of the integration by parts formula in \eqref{eq:ibp} yields $\ind_f(u,u)=\indo_f(u,u)$. This proves the claim.
\end{proof}

Following the terminology in \cite{bdce} we call \emph{$f$-Jacobi operator} of $\Sg$ to the second order linear operator $\mathcal{L}_f$ in \eqref{eq:hfprima}. Note that $\mathcal{L}_f$ coincides with the derivative of the $f$-mean curvature function along the variation. The \emph{$f$-index form} of $\Sg$ is the symmetric bilinear form $\indo_f$ on $C^\infty_0(\Sg)$ defined in \eqref{eq:indo}. By using formula \eqref{eq:ibp} we get $\ind_f(v,w)=\indo_f(v,w)$ for any $u,v\in C^\infty_0(\Sg)$. In particular, the symmetry of $\indo_f$ gives us the equality
\begin{equation*}
\int_\Sg \left\{v\,\mathcal{L}_f(w)-w\,\mathcal{L}_f(v)\right\}da_f
=\int_{\ptl\Sg}\left\{w\,\frac{\ptl v}{\ptl\nu}-v\,\frac{\ptl w}{\ptl\nu}\right\}dl_f,
\end{equation*} 
for any two functions $v,w\in C^\infty_0(\Sg)$.

Let $\Sg$ be an $f$-stationary hypersurface of constant $f$-mean curvature $H_f$. We say that $\Sg$ is \emph{strongly $f$-stable} if we have $(A_f+H_f\,V_f)''(0)\geq 0$ for any variation of $\Sg$. We say that $\Sg$ is \emph{$f$-stable} if $A_f''(0)\geq 0$ for any volume-preserving variation. For a strongly $f$-stationary hypersurface, to be strongly $f$-stable is the analogous property satisfied by free boundary stable minimal hypersurfaces in Riemannian manifolds.

By using the orthogonality condition in Corollary~\ref{cor:stationary} (ii) and the arguments given in \cite[Lem.~(2.2)]{bdce}, any function $u\in C^\infty_0(\Sg)$ with $\int_\Sg u\,da_f=0$ is the normal component of the velocity vector associated to a volume-preserving variation of $\Sg$. As a consequence, we can deduce the following result from Proposition~\ref{prop:2nd}.

\begin{corollary}
\label{cor:stable}
Let $M$ be a smooth oriented Riemannian manifold endowed with a density $f=e^\psi$. Consider a smooth oriented $f$-stationary hypersurface $\Sg$ immersed in $M$ with $\emph{int}(\Sg)\sub\emph{int}(M)$ and $\ptl\Sg\sub\ptl M$. Let $\indo_f$ be the index form of $\Sg$ defined in \eqref{eq:indo}. Then, we have
\begin{itemize}
\item[(i)] $\Sg$ is strongly $f$-stable if and only if $\indo_f(u,u)\geq 0$, for any $u\in C^\infty_0(\Sg)$.
\item[(ii)] $\Sg$ is $f$-stable if and only if $\indo_f(u,u)\geq 0$, for any $u\in C^\infty_0(\Sg)$ with $\int_\Sg u\,da_f=0$.
\end{itemize} 
\end{corollary}

\begin{example}
\label{ex:2}
Let $M=0\cone\mathcal{D}$ be a cone over a smooth region $\mathcal{D}$ of the unit sphere $\sph^n$ in $\rrn$. Suppose that $f=e^\psi$ is a smooth radial density on $M$ and denote by $\Sg$ the intersection with $M$ of a round sphere of radius $r$ centered at $0$. By following the computations in \cite[Thm.~3.10]{rcbm} we see that the $f$-index form associated to $\Sg$ is given by
\[
\indo_f(u,u)=f(r)\left[\int_\Sg\left(|\nabla_\Sg u|^2-|\sg|^2\,u^2\right)da
-\int_{\ptl\Sg}\text{II}(N,N)\,u^2\,dl+\int_\Sg\psi''(r)\,u^2\,da\right].
\]
In general, we cannot expect that $\Sg$ is $f$-stable. For example, if $M$ is a half-space and $f$ is strictly log-concave, then we can move $\Sg$ by translations along a fixed direction of $\ptl M$ to find a function $u$ with $\int_\Sg u\,da_f=0$ and $\indo_f(u,u)<0$. However, if $M$ is convex and $f$ is log-convex, then $\Sg$ is $f$-stable. To see this we take a function $u\in C^\infty(\Sg)$ with $\int_\Sg u\,da_f=0$. On the one hand, we have $\int_\Sg u\,da=0$ and we can use that $\Sg$ is free boundary stable in $M$ with Euclidean density \cite{cones} to deduce that the sum of the two first terms in $\indo_f(u,u)$ is nonnegative.  On the other hand, the log-convexity of $f$ implies $\psi''(r)\geq 0$. So we get $\indo_f(u,u)\geq 0$, as claimed.  As in \cite{rcbm} this fact might suggest that in a Euclidean solid convex cone endowed with a smooth, radial, log-convex density, any round sphere centered at the origin intersected with the cone minimizes the interior weighted area among all the hypersurfaces in the cone enclosing the same weighted volume.
\end{example}

\begin{example}
\label{ex:3}
Let $M$ be a Euclidean solid cone endowed with a $k$-homogeneous density $f=e^\psi$. In \cite[Ex.~4.7]{homostable} it was shown that the intersection with $M$ of a round sphere centered at $0$ is strongly $f$-stable if and only if $k\leq -n$. By assuming that $M$ is convex and the Bakry-\'Emery-Ricci tensor satisfies $\ric_f\geq (1/k)(d\psi\otimes d\psi)$ for some $k>0$, it was proved in \cite[Re.~1.5]{cabre3} that such spherical caps are $f$-stable (they are minimizers of the interior weighted area for fixed weighted volume). In fact, in \cite[Thm.~5.11]{homostable} it is shown that these are the unique compact $f$-stable hypersurfaces in $M$. Unduloidal examples of $f$-stable curves inside planar sectors with density $f(p)=|p|^k$, $k>0$, appear in \cite{dhht}. 
\end{example}

\section{Topology and rigidity of compact strongly stable hypersurfaces}
\label{sec:main}
\setcounter{equation}{0}

In this section we prove the main results of the paper. We will obtain topological and geometrical restrictions for strongly $f$-stable hypersurfaces under certain curvature and boundary assumptions on the ambient manifold. Our statements and proofs are inspired by previous results for the Riemannian case, see \cite{cai-galloway}, \cite{bbn}, \cite{nunes}, \cite{micallef}, \cite{fbr}, \cite{ambrozio}, and for hypersurfaces with empty boundary in manifolds with density, see \cite{fan}, \cite{liu} and \cite{espinar}. 

We will use the same notation as in the previous section. For a given oriented hypersurface $\Sg$ in a Riemannian manifold $M$ with boundary $\ptl M$, we denote by $N$, $\nu$ and $\sg$ the unit normal to $\Sg$, the conormal vector to $\ptl\Sg$, and the second fundamental form of $\Sg$, respectively.  We denote by $\xi$ and $\text{II}$ the inner unit normal to $\ptl M$ and the second fundamental form of $\ptl M$ with respect to $\xi$.

\subsection{Non-negative Bakry-\'Emery-Ricci curvature and locally convex boundary}
\label{subsec:ricci}

Let $M$ be a smooth oriented Riemannian manifold endowed with a density $f=e^\psi$. Given a point $p\in M$, we define the \emph{Bakry-\'Emery-Ricci curvature} of $M$ at $p$ as the quadratic form $v\in T_pM\mapsto\ric_f(v,v)$, where $\ric_f$ is the $2$-tensor in \eqref{eq:fricci}. Observe that, for an $f$-stationary hypersurface $\Sg$ immersed in $M$, the $f$-index form $\indo_f$ introduced in \eqref{eq:indo} involves the normal Bakry-\'Emery-Ricci curvature $\ric_f(N,N)$. As a consequence, if $\ric_f(N,N)\geq 0$ on $\Sg$ and $\ptl M$ is \emph{locally convex}, i.e., $\text{II}$ is always positive semidefinite, then $\indo_f(u,u)$ contains non-positive terms, and so the stability condition in Corollary~\ref{cor:stable} (i) becomes more restrictive. In fact, by inserting $u=1$ inside $\indo_f$ we can prove the following simple but interesting result.

\begin{lemma}
\label{lem:easy}
Let $M$ be a smooth oriented Riemannian manifold with locally convex boundary, and endowed with a density $f=e^\psi$ of non-negative Bakry-\'Emery-Ricci curvature. Consider a smooth, compact, oriented, $f$-stationary hypersurface $\Sg$ immersed in $M$ with $\emph{int}(\Sg)\sub\emph{int}(M)$ and $\ptl\Sg\sub\ptl M$. Then, $\Sg$ is strongly $f$-stable if and only if $\Sg$ is totally geodesic, $\emph{Ric}_f(N,N)=0$ on $\Sg$, and $\emph{II}(N,N)=0$ along $\ptl\Sg$.
\end{lemma}

The simplest case where strongly $f$-stable hypersurfaces in the conditions of Lemma~\ref{lem:easy} appear is the Riemannian product $\rr\times\Sg$, where $\Sg$ is a compact manifold of non-negative Ricci curvature. In fact, for any density $f=e^\psi$ with $\psi(s,p)=as+b$, the horizontal slices $\{s\}\times\Sg$ are strongly $f$-stable hypersurfaces.  As it is shown in \cite[Sect.~1]{micallef} other examples exist where the ambient manifold $M$ does not split along $\Sg$. However, by assuming that $\Sg$ is embedded and locally weighted area-minimizing we can obtain a rigidity result in the same spirit of the one proved by Liu \cite[Thm.~1]{liu} for hypersurfaces with empty boundary. Before stating the theorem we need two definitions. We say that $\Sg$ is \emph{locally weighted area-minimizing} if for any variation $\varphi:(-\eps,\eps)\times\Sg\to M$ of $\Sg$, the associated weighted area functional satisfies $A_f(0)\leq A_f(s)$ for any $s$ in a small open interval containing the origin. We say that $\Sg$ \emph{minimizes the weighted area in its isotopy class} if, for any variation $\varphi$ of $\Sg$ such that the maps $\varphi_s:\Sg\to\Sg_s$ are diffeomorphisms, then $A_f(0)\leq A_f(s)$, for any $s\in (-\eps,\eps)$.

\begin{theorem}
\label{th:ricci}
Let $M$ be a smooth oriented Riemannian manifold with locally convex boundary, and endowed with a density $f=e^\psi$ of non-negative Bakry-\'Emery-Ricci curvature. Suppose that $\Sg$ is a smooth, compact, oriented, locally weighted area-minimizing hypersurface embedded in $M$ with $\emph{int}(\Sg)\sub\emph{int}(M)$ and $\ptl\Sg\sub\ptl M$. Then, $\Sg$ is totally geodesic, and there is an open neighborhood of $\Sg$ in $M$ which is isometric to a Riemannian product $(-\eps_0,\eps_0)\times\Sg$. Moreover, if $M$ is complete and $\Sg$ minimizes the weighted area in its isotopy class, then the Riemannian product $\rr\times\Sg$ is an isometric covering of $M$.
\end{theorem}

Liu's proof of Theorem~\ref{th:ricci} when $\ptl\Sg=\emptyset$ uses the second variation formula for the area and the fact that $\Sg$ is locally area-minimizing to deduce that the local flow of normal geodesics leaving from $\Sg$ keeps constant the weighted area. However, in the case $\ptl\Sg\neq\emptyset$, we must deform $\Sg$ in a different way since a normal geodesic starting from $\ptl\Sg$ may leave the manifold $M$. This deformation is carried out in the next proposition for hypersurfaces satisfying the conclusions of Lemma~\ref{lem:easy}. 

\begin{proposition}
\label{prop:deformation}
Let $M$ be a smooth oriented Riemannian manifold endowed with a density $f=e^\psi$. Consider a smooth, compact, oriented, $f$-stationary hypersurface $\Sg$ immersed in $M$ with $\emph{int}(\Sg)\sub\emph{int}(M)$ and non-empty boundary $\ptl\Sg\sub\ptl M$. If $\Sg$ is totally geodesic, $\emph{Ric}_f(N,N)=0$ on $\Sg$ and $\emph{II}(N,N)=0$ along $\ptl\Sg$, then there is a variation $\varphi:(-\eps_0,\eps_0)\times\Sg\to M$ of $\Sg$ with velocity vector $X=N$ on $\Sg$, and such that any hypersurface $\Sg_s:=\varphi_s(\Sg)$ is $f$-stationary. Moreover, if $\Sg$ is embedded, then $\Om:=\varphi\big((-\eps_0,\eps_0)\times\Sg\big)$ is an open neighborhood of $\Sg$ in $M$ and $\varphi:(-\eps_0,\eps_0)\times\Sg\to\Om$ is a diffeomorphism.
\end{proposition}

\begin{proof}
We adapt to weighted manifolds the arguments in \cite[Prop.~10]{ambrozio}. Fix a smooth vector field $Y$ on $M$ such that $Y=N$ on $\Sg$ and $Y$ is tangent to $\ptl M$. We denote by $\{\phi_s\}$ the associated one-parameter group of diffeomorphisms. For fixed $\alpha\in (0,1)$ we can find numbers $\tau>0$ and $\delta>0$ such that, for any pair $(s,u)$ with $s\in (-\tau,\tau)$ and $u$ in the open ball $B_\delta(0)$ of the H\"older space $C^{2,\alpha}(\Sg)$, the set $\Sg_{u+s}:=\{\phi_{u(p)+s}(p)\,;\,p\in\Sg\}$ is an immersed $C^{2,\alpha}$ hypersurface with $\text{int}(\Sg_{u+s})\sub\text{int}(M)$ and $\ptl\Sg_{u+s}\sub\ptl M$. Moreover, if $\Sg$ is embedded, then $\Sg_{u+s}$ is also embedded. Let us denote $E:=\{u\in C^{2,\alpha}(\Sg)\,;\,\int_\Sg u\,da_f=0\}$ and $F:=\{u\in C^{0,\alpha}(\Sg)\,;\,\int_\Sg u\,da_f=0\}$. Thus, we have a well-defined map $\Phi:(-\tau,\tau)\times (B_\delta(0)\cap E)\to F\times C^{1,\alpha}(\ptl\Sg)$ given by
\[
\Phi(s,u):=\left((H_f)_{u+s}-\frac{1}{A_f(\Sg)}\int_\Sg (H_f)_{u+s}\,\,da_f,\escpr{\xi,N_{u+s}}\right),
\]
where $N_{u+s}$ and $(H_f)_{u+s}$ denote the unit normal and the $f$-mean curvature of $\Sg_{u+s}$, respectively. Note that $\Phi(0,0)=(0,0)$ since $\Sg$ is $f$-stationary. In fact, a hypersurface $\Sg_{u+s}$ will be also $f$-stationary if and only if $\Phi(s,u)=(0,0)$. So, we try to apply the implicit function theorem to $\Phi$ at $(0,0)$. For any $w\in E$, we can construct the variation $\eta(s,p):=\phi_{sw(p)}(p)$, whose velocity vector equals $wN$ on $\Sg$. By using formulas \eqref{eq:hfprima} and \eqref{eq:mimo}, equalities $\ric_f(N,N)=|\sg|^2=\text{II}(N,N)=0$, and the divergence theorem in \eqref{eq:divthsup}, we get 
\begin{align*}
(d\Phi)_{(0,0)}(0,w)&=\left(\mathcal{L}_f(w)
-\frac{1}{A_f(\Sg)}\int_\Sg\mathcal{L}_f(w)\,da_f,-\frac{\ptl w}{\ptl\nu}
\right)
\\
&=\left(\Delta_{\Sg,f}(w)+\frac{1}{A_f(\Sg)}\int_{\ptl\Sg}\frac{\ptl w}{\ptl\nu}\,dl_f,-\frac{\ptl w}{\ptl\nu}\right),
\end{align*}
where $\Delta_{\Sg,f}$ is the $f$-Laplacian defined in \eqref{eq:deltaf}. Let us see that $(d\Phi)_{(0,0)}:\{0\}\times E\to F\times C^{1,\alpha}(\ptl\Sg)$ is an isomorphism. Take functions $h\in F$ and $k\in C^{1,\alpha}(\ptl\Sg)$. Then we have $\int_{\Sg}(h+\beta)\,da_f=\int_{\ptl\Sg}k\,dl_f$, where $\beta:=A_f(\Sg)^{-1}\int_{\ptl\Sg}k\,dl_f$. Now, we can apply existence and uniqueness of solutions for Poisson type equations with Neumann boundary conditions, see \cite[Sect.~3.3]{equations} and \cite{nardi}, to conclude that there is a unique function $w\in E$ solving the problem $\Delta_{\Sg,f}(w)=h+\beta$ on $\Sg$ and $\ptl w/\ptl\nu=-k$ along $\ptl\Sg$. As a consequence $(d\Phi)_{(0,0)}(0,w)=(h,k)$. Moreover, $w$ is unique in $E$ satisfying this property. 

Hence, we can find $\eps_0>0$ and a curve $u:(-\eps_0,\eps_0)\to B_\delta(0)\cap E$ such that $u(0)=0$ and $\Phi(s,u(s))=\Phi(0,0)=(0,0)$, for any $s\in (-\eps_0,\eps_0)$. In particular, any hypersurface $\Sg_{u(s)+s}$ is $f$-stationary. Finally, we define $\varphi:(-\eps_0,\eps_0)\times\Sg\to M$ as $\varphi(s,p):=\phi_{\mu(s,p)}(p)$, where $\mu(s,p):=s+u(s)(p)$. Note that $\varphi(0,p)=\phi_0(p)=p$ for any $p\in\Sg$, and so $\varphi$ is a variation of $\Sg$. For any $s\in (-\eps_0,\eps_0)$ we have $\Sg_s:=\varphi_s(\Sg)=\Sg_{u(s)+s}$, which is $f$-stationary. The velocity vector equals $X_p=(\ptl\mu/\ptl s)(0,p)\,N_p$. By differentiating with respect to $s$ in equality $\Phi(s,u(s))=(0,0)$, and using again \eqref{eq:hfprima} and \eqref{eq:mimo}, we deduce that $(\ptl\mu/\ptl s)(0,p)$ solves the problem $\Delta_{\Sg,f}(w)=0$ on $\Sg$ with $\ptl w/\ptl\nu=0$ along $\ptl\Sg$. Hence $(\ptl\mu/\ptl s)(0,p)$ is constant as a function of $p\in\Sg$. From equality $0=\int_\Sg u(s)\,da_f=\int_\Sg(\mu(s,p)-s)\,da_f$ we get $(\ptl\mu/\ptl s)(0,p)=1$, so that $X=N$ on $\Sg$. To finish the proof we apply the inverse function theorem and we find a smaller $\eps_0>0$ such that $\Om:=\varphi\big((-\eps_0,\eps_0)\times\Sg\big)$ is open in $M$ and the map $\varphi:(-\eps_0,\eps_0)\times\Sg\to\Om$ is a diffeomorphism.
\end{proof}

Now, we are ready to prove Theorem~\ref{th:ricci}.

\begin{proof}[Proof of Theorem~\ref{th:ricci}]
First note that $\Sg$ is strongly $f$-stationary and strongly $f$-stable. So, we can deduce by Corollary~\ref{cor:stationary} and Lemma~\ref{lem:easy} that the $f$-mean curvature $H_f$ of $\Sg$ vanishes, $\Sg$ is totally geodesic, and equalities $\ric_f(N,N)=\text{II}(N,N)=0$ hold. In particular, we can apply Proposition~\ref{prop:deformation} to obtain a variation $\varphi:(-s_0,s_0)\times\Sg\to M$ of $\Sg$ such that $X=N$ on $\Sg$, the hypersurfaces $\Sg_s:=\varphi_s(\Sg)$ are all $f$-stationary, and $\Om:=\varphi\big((-s_0,s_0)\times\Sg\big)$ is an open neighborhood of $\Sg$ in $M$ diffeomorphic to $(-s_0,s_0)\times\Sg$. 

Let us prove that the variation $\varphi$ does not increase the area, i.e., $A_f(s)\leq A_f(0)$ for any $s$ in a small open interval containing $0$. We will use the subscript $s$ to denote geometric functions and vectors associated to $\Sg_s$. We define $(X_s)_p:=(\ptl\varphi/\ptl s)(s,p)$ and $u_s:=\escpr{X_s,N_s}$. Since $u_0=1$, we can assume by continuity that $u_s>0$ on $\Sg_s$ for any $s\in (-s_0,s_0)$. On the other hand, as any $\Sg_s$ is $f$-stationary, we deduce by Corollary~\ref{cor:stationary} (ii) that $\escpr{X_s,\nu_s}=0$ along $\ptl\Sg_s$. In particular, equation~\eqref{eq:bd} yields $\ptl u_s/\ptl\nu_s+\text{II}(N_s,N_s)\,u_s=0$ along $\ptl\Sg_s$. Let $H_f(s)$ be the function that maps any $s\in (-s_0,s_0)$ to the constant $f$-mean curvature of $\Sg_s$. In order to show that $\varphi$ does not increase the area it suffices, by Lemma~\ref{lem:1st}, to see that $H'_f(s)\geq 0$ for any $s\in (-s_0,s_0)$. Observe that $H_f'(s)=(\mathcal{L}_f)_s(u_s)$, where $(\mathcal{L}_f)_s$ is the $f$-Jacobi operator on $\Sg_s$ defined in \eqref{eq:hfprima}. As a consequence
\begin{align*}
H_f'(s)\,A_f(s)&=\int_{\Sg_s}H_f'(s)\,(da_f)_s=\int_{\Sg_s}(\mathcal{L}_f)_s(u_s)\,(da_f)_s
\\
&=\int_{\Sg_s}\left\{\Delta_{\Sg_s,f}(u_s)+\big(\ric_f(N_s,N_s)+|\sg_s|^2\big)\,u_s\right\}(da_f)_s
\\
&=\int_{\ptl\Sg_s}\text{II}(N_s,N_s)\,u_s\,(dl_f)_s
+\int_{\Sg_s}\big(\ric_f(N_s,N_s)+|\sg_s|^2\big)\,u_s\,(da_f)_s,
\end{align*}
where we have used \eqref{eq:divthsup} and that $\ptl u_s/\ptl\nu_s=-\text{II}(N_s,N_s)\,u_s$. Therefore, the local convexity of $\ptl M$ and the nonnegativity of the Bakry-\'Emery-Ricci curvature give us $H_f'(s)\geq 0$. Moreover, if equality holds for some $s\in(-s_0,s_0)$, then $\Sg_s$ is totally geodesic, $\ric_f(N_s,N_s)=0$ on $\Sg_s$ and $\text{II}(N_s,N_s)=0$ along $\ptl\Sg_s$. In particular, the function $u_s$ solves the problem $\Delta_{\Sg_s,f}(u_s)=0$ on $\Sg$ with $\ptl u_s/\ptl\nu_s=0$ along $\ptl\Sg_s$, so that $u_s$ must be constant on $\Sg_s$.

Now, we can prove the conclusions of the theorem. By using that $\Sg$ is locally weighted area-minimizing, we get $A_f(s)=A_f(0)$ for any $s$ in a small open interval $J$ containing $0$. By Lemma~\ref{lem:1st}, this implies $H_f(s)\equiv 0$, and so $H'_f(s)=0$ for any $s\in J$. From the previous discussion we deduce that $\Sg_s$ is totally geodesic and $u_s$ is constant on $\Sg_s$ for any $s\in J$. By taking into account equation \eqref{eq:normal} and that $\Sg_s$ is totally geodesic, we infer that $N_s$ is a parallel vector field defined on $\Om$. So, the integral curves of $N_s$ are geodesics, and we can find $\eps_0>0$, and an open neighborhood $U_0\subeq\Om$ of $\Sg$ in $M$, such that the flow by normal geodesics $F:(-\eps_0,\eps_0)\times\Sg\to U_0$ given by 
$F(s,p):=\exp_p(sN_p)$ is a diffeomorphism. Moreover, $F$ is an isometry since $N_s$ is a Killing field.

Finally, let us assume that $M$ is complete and $\Sg$ minimizes the weighted area in its isotopy class. Note that the flow by normal geodesics $F$ is well-defined on $\rr\times\Sg$. Let $s_\infty$ be the supremum of the set $B$ of the numbers $s>0$ such that $F:[-s,s]\times\Sg\to M$ is an isometry onto its image. Suppose $s_\infty<+\infty$ and denote $\Sg_{\pm\infty}=F(\{s_{\pm\infty}\}\times\Sg)$. As in the first part of the proof, we can see that the variation $F$ does not increase the area. Thus, we would get $A_f(\Sg)=A_f(\Sg_{\pm\infty})$ by the minimization property of $\Sg$. Hence, the hypersurfaces $\Sg_{\pm\infty}$ would be locally weighted area-minimizing, and we may use the first conclusion of the theorem to find $\beta>0$ such that $s_\infty+\beta\in B$, a contradiction. So, we have $s_\infty=+\infty$. As a consequence $F:\rr\times\Sg\to M$ is a local isometry and, in particular, a covering map. This completes the proof. 
\end{proof}

\subsection{Non-negative Perelman scalar curvature and $f$-mean convex boundary}
\label{subsec:scalar}

Here we provide topological estimates and rigidity results for strongly $f$-stable surfaces in weighted manifolds by assuming a certain condition on the Perelman scalar curvature and weigh\-ted mean convexity of the boundary. We restrict ourselves to dimension $3$ since the Gauss-Bonnet theorem will be a key ingredient in our proofs.

We first recall some notation and introduce a definition. Let $M$ be a smooth oriented Riemannian $3$-manifold endowed with a density $f=e^\psi$. Recall that the Perelman scalar curvature is the function defined in \eqref{eq:fscalar} by $S_f:=S-2\Delta\psi-|\nabla\psi|^2$. We will say that the boundary $\ptl M$ is \emph{$f$-mean convex} if the $f$-mean curvature $(H_f)_{\ptl M}$ of $\ptl M$ introduced in \eqref{eq:fmc} is nonnegative when computed with respect to the inner unit normal $\xi$.

Now we can prove a first result, where we obtain a topological restriction for strongly $f$-stable surfaces where $S_f+H^2_f\geq 0$.

\begin{theorem}
\label{th:main}
Let $M$ be a smooth oriented Riemannian $3$-manifold endowed with a density $f=e^\psi$ such that $\ptl M$ is $f$-mean convex. Consider a smooth, compact, connected, oriented, $f$-stationary surface $\Sg$ immersed in $M$ with $\emph{int}(\Sg)\sub\emph{int}(M)$ and $\ptl\Sg\sub\ptl M$. If $\Sg$ is strongly $f$-stable and $S_f+H^2_f\geq 0$, then $\Sg$ has non-negative Euler characteristic. More precisely, we have 
\begin{enumerate}
\item[(i)] If $\ptl\Sg=\emptyset$, then $\Sg$ is a sphere or a torus.
\item[(ii)] If $\ptl\Sg\neq\emptyset$, then $\Sg$ is a disk or a cylinder. 
\end{enumerate}  
If the Euler characteristic vanishes, then $\Sg$ is flat and totally geodesic, the density $f$ is constant on $\Sg$, and $S_f+H^2_f=\emph{Ric}_f(N,N)=0$ on $\Sg$. Moreover, if $\ptl\Sg\neq\emptyset$, then it consists of two closed geodesics in $M$ where $\emph{II}(N,N)=(H_f)_{\ptl M}=0$.
\end{theorem}

\begin{proof}
We first obtain two identities, one for the interior of $\Sg$ and another one for the boundary $\ptl\Sg$, that will be key ingredients to prove the claim. From the Gauss equation we get the following rearrangement already described in the proof of \cite[Thm.~5.1]{schoen-yau}
\[
\ric(N,N)+|\sg|^2=\frac{1}{2}\,S+2H^2+\frac{1}{2}\,|\sg|^2-K,
\]
where $H$ and $K$ denote the Riemannian mean curvature and the Gauss curvature of $\Sg$, respectively. Note also that
\[
\Delta\psi=\divv_\Sg\nabla\psi+(\nabla^2\psi)(N,N)=\Delta_\Sg\psi-2H\,\escpr{\nabla\psi,N}+(\nabla^2\psi)(N,N),
\]
where we have used $\nabla\psi=\nabla_\Sg\psi+\escpr{\nabla\psi,N}N$ and $\divv_\Sg N=-2H$ to obtain the second equality. Combining the two previous equations together with \eqref{eq:fricci}, \eqref{eq:fmc} and \eqref{eq:fscalar}, we deduce
\begin{equation}
\label{eq:yo1}
\ric_f(N,N)+|\sg|^2=\frac{1}{2}\,(S_f+H^2_f)+\frac{1}{2}\,(|\sg|^2+|\nabla_\Sg\psi|^2)-K+\Delta_\Sg\psi\quad\text{ on } \Sg.
\end{equation}
On the other hand, the fact that $\Sg$ is $f$-stationary implies, by Corollary~\ref{cor:stationary} (ii), that $\Sg$ meets $\ptl M$ orthogonally along the boundary curves. Thus, the inner unit normal $\xi$ of $\ptl M$ coincides with the conormal vector $\nu$. As a consequence $\text{II}(T,T)=h$, where $T$ is a unit tangent vector to $\ptl\Sg$ and $h$ is the geodesic curvature of $\ptl\Sg$ in $\Sg$. Therefore, we have 
\begin{equation}
\label{eq:yo2}
\text{II}(N,N)=2H_{\ptl M}-h \quad\text{ along }\ptl\Sg,
\end{equation}
where $H_{\ptl M}$ is the Riemannian mean curvature of $\ptl M$ with respect to $\xi$.

Now, we take the function $u:=1/\sqrt{f}$. The strong $f$-stability of $\Sg$ together with Corollary~\ref{cor:stable} (i) and the definition of $f$-index form in \eqref{eq:indo} gives us
\[
0\leq\indo_f(u,u)=\int_\Sg\left\{|\nabla_\Sg u|^2-\big(\ric_f(N,N)+|\sg|^2\big)\,u^2\right\}da_f-\int_{\ptl\Sg}\text{II}(N,N)\,u^2\,dl_f,
\]
where we understand that the boundary integral vanishes provided $\ptl\Sg=\emptyset$. Note that $u^2\,da_f=da$, $u^2\,dl_f=dl$ and $|\nabla_\Sg u|^2\,da_f=(1/4)\,|\nabla_\Sg\psi|^2\,da$. Therefore, by substituting equalities \eqref{eq:yo1} and \eqref{eq:yo2} into the previous expression, and taking into account that $\int_\Sg\Delta_\Sg\psi\,da=-\int_{\ptl\Sg}\escpr{\nabla\psi,\nu}\,dl$, we conclude
\begin{align}
\label{eq:yo3}
0\leq\indo_f(u,u)&=\int_\Sg\left\{\frac{-1}{4}\,|\nabla_\Sg\psi|^2-\frac{1}{2}\,(S_f+H_f^2)-\frac{1}{2}\,|\sg|^2\right\}da+\int_\Sg K\,da
\\
\nonumber
&+\int_{\ptl\Sg}\left\{-2H_{\ptl M}+\escpr{\nabla\psi,\nu}\right\}dl
+\int_{\ptl\Sg}h\,dl
\\
\nonumber
&\leq\int_{\Sg} K\,da-\int_{\ptl\Sg}(H_f)_{\ptl M}\,dl+2\pi\,\chi-\int_\Sg K\,da
\\
\nonumber
&\leq 2\pi\chi=2\pi\,(2-2g-m),
\end{align}
where we have used that $S_f+H_f^2\geq 0$ on $\Sg$, the Gauss-Bonnet theorem, and the $f$-mean convexity of $\ptl M$. We have also denoted by $\chi$, $g$ and $m$ the Euler characteristic, the genus and the number of boundary components of $\Sg$, respectively. From the previous inequality we easily deduce statements (i) and (ii) of the theorem.

Now, suppose $\chi=0$. Then, equality holds in \eqref{eq:yo3}, and so $\Sg$ is a totally geodesic surface such that $|\nabla_\Sg\psi|^2=0$, $S_f+H_f^2=0$ and $(H_f)_{\ptl M}=H_{\ptl M}=0$. Moreover, we also have $\indo_f(u,u)=0$. For any function $v\in C^\infty(\Sg)$ and any $s\in\rr$ we get
\begin{align*}
0\leq\indo_f(u+sv,u+sv)&=\indo_f(u,u)+2s\,\indo_f(u,v)+s^2\,\indo_f(v,v)
\\
&=2s\,\indo_f(u,v)+s^2\,\indo_f(v,v),
\end{align*}
since $\Sg$ is strongly $f$-stable and $\indo_f(u,u)=0$. This implies that $\indo_f(u,v)=0$ for any $v\in C^\infty(\Sg)$. Let $\mathcal{L}_f$ be the $f$-Jacobi operator defined in \eqref{eq:hfprima}. After applying the integration by parts formula in \eqref{eq:ibp}, we obtain
\[
0=\indo_f(v,u)=\ind_f(v,u)=-\int_\Sg v\,\mathcal{L}_f(u)\,da_f
-\int_{\ptl\Sg} v\left\{\frac{\ptl u}{\ptl\nu}+\text{II}(N,N)\,u\right\}dl_f,
\]
for any $v\in C^\infty(\Sg)$. From here we deduce $\mathcal{L}_f(u)=0$ on $\Sg$ and $\ptl u/\ptl\nu+\text{II}(N,N)\,u=0$ along $\ptl\Sg$. By taking into account that $u>0$ and that $\nabla_\Sg u=-(\nabla_\Sg\psi)/(2\sqrt{f})=0$, we conclude that $\ric_f(N,N)=0$ on $\Sg$ and $\text{II}(N,N)=0$ along $\ptl\Sg$. Finally, equations \eqref{eq:yo1} and \eqref{eq:yo2} give $K=0$ on $\Sg$ and $h=0$ along $\ptl\Sg$. This completes the proof. 
\end{proof}

\begin{example}
Consider the manifold $M_r=\{p\in\rr^{3}\,;\,|p|\geq r\}$ endowed with the Euclidean metric and the radial $k$-homogeneous density $f(p)=|p|^k$. It is easy to check, see the computations in \cite[Ex.~4.3]{homostable}, that the $f$-mean curvature of $\ptl M_r$ with respect to the inner unit normal equals $-(k+2)/r$. Moreover, the Perelman scalar curvature is given by $S_f(p)=-k\,|p|^{-2}\,(k+2)$, see \cite[Lem.~3.6]{homostable} for details. Hence, in the case $k=-2$, we have that $M_r$ has $f$-mean convex boundary and vanishing Perelman scalar curvature. By applying Theorem~\ref{th:main} we conclude that any compact strongly $f$-stable hypersurface in $M_r$ is topologically a sphere or a disk.
\end{example}

\begin{example}
Theorem~\ref{th:main} shows that the existence of strongly $f$-stable tori or cylinders is very restrictive. However, we can find some situations where they appear. Take the manifold $M=\rr\times\sph^1\times [-1,1]$ endowed with the Riemannian product metric and the density $f(s,\theta,t)=e^s$.  It is easy to check that $M$ has $f$-mean convex boundary and that any horizontal cylinder $\Sg=\{s\}\times\sph^1\times [-1,1]$ is $f$-stationary with $S_f+H^2_f=-1+1=0$. Moreover, $\Sg$ is totally geodesic with $\ric_f(N,N)=0$ and $\text{II}(N,N)=0$, which implies by \eqref{eq:indo} and Corollary~\ref{cor:stable} (i) that $\Sg$ is strongly $f$-stable. Similarly, in the Riemannian product $M=\rr\times\sph^1\times\sph^1$ with density $f(s,\theta,t)=e^{-s}$ any horizontal torus $\{s\}\times\sph^1\times\sph^1$ is strongly $f$-stable.
\end{example}

Observe that Theorem~\ref{th:main} applies when $\ptl\Sg=\emptyset$, $S_f\geq 0$ on $M$, and $H_f=0$ on $\Sg$, thus generalizing previous results in \cite[Thm.~2.1]{fan} and \cite[Prop.~8.1]{espinar}. In the next corollary we particularize Theorem~\ref{th:main} when $S_f\geq 0$ and $\ptl\Sg\neq\emptyset$. This provides an extension of the results in \cite[Thm.~1.2]{fbr} and \cite[Prop.~6]{ambrozio} for stable free boundary minimal surfaces in Riemannian $3$-manifolds. 

\begin{corollary}
\label{cor:main}
Let $M$ be a smooth oriented Riemannian $3$-manifold endowed with a density $f=e^\psi$ such that $S_f\geq 0$ on $M$ and $\ptl M$ is $f$-mean convex. Consider a smooth, compact, connected, oriented, $f$-stationary surface $\Sg$ immersed in $M$ with $\emph{int}(\Sg)\sub\emph{int}(M)$ and non-empty boundary $\ptl\Sg\sub\ptl M$. If $\Sg$ is strongly $f$-stable, then $\Sg$ is either a disk or a totally geodesic flat cylinder. In the last case, the density $f$ is constant on $\Sg$, $S_f=H_f=\emph{Ric}_f(N,N)=0$ on $\Sg$, $\ptl\Sg$ consists of two closed geodesics in $M$, and $\emph{II}(N,N)=(H_f)_{\ptl M}=0$ along $\ptl\Sg$.
\end{corollary}

The existence of strongly stable cylinders in the conditions of the previous corollary cannot be discarded. The model situation where they appear is a Riemannian product $\rr\times\Sg$ with constant density, where $\Sg$ is the compact cylinder $\sph^1\times [a,b]$ endowed with the Euclidean metric of $\rr^3$. In fact, any horizontal slice $\{s\}\times\Sg$ provides a strongly stable cylinder. This is not the unique example we may give, in the sense that the existence of such cylinders does not imply that $M$ locally splits as a Riemannian product, see \cite[Sect.~1]{micallef}. However, by assuming that the cylinder is embedded and locally area-minimizing, we can obtain a rigidity result in the same spirit of \cite[Thm~1]{cai-galloway}, \cite[Thm.~8.1]{espinar} and \cite[Thm.~7]{ambrozio}. 

\begin{theorem}
\label{th:main2}
Let $M$ be a smooth oriented Riemannian $3$-manifold endowed with a density $f=e^\psi$ such that $S_f\geq 0$ on $M$ and $\ptl M$ is $f$-mean convex. Suppose that there is a locally weighted area-minimizing smooth oriented cylinder $\Sg$ embedded in $M$ with $\emph{int}(\Sg)\sub\emph{int}(M)$ and $\ptl\Sg\sub\ptl M$. 
Then, $\Sg$ is flat with geodesic boundary, and there is an open neighborhood of $\Sg$ in $M$ which is isometric to a Riemannian product $(-\eps_0,\eps_0)\times\Sg$ with constant density. Moreover, if $M$ is complete and $\Sg$ minimizes the weighted area in its isotopy class, then the density $f$ is constant in $M$ and the Riemannian product $\rr\times\Sg$ is an isometric covering of $M$.
\end{theorem}

\begin{proof}
The scheme of the proof is the same as in Theorem~\ref{th:ricci}. First note that $\Sg$ is strongly $f$-stationary and strongly $f$-stable. So, we can deduce by Corollary~\ref{cor:main} that the $f$-mean curvature $H_f$ of $\Sg$ vanishes, the density $f$ is constant on $\Sg$, the surface $\Sg$ is totally geodesic and flat with geodesic boundary, and equalities $\ric_f(N,N)=\text{II}(N,N)=0$ hold. In particular, we can apply Proposition~\ref{prop:deformation} to obtain a variation $\varphi:(-s_0,s_0)\times\Sg\to M$ of $\Sg$ such that $X=N$ on $\Sg$, the hypersurfaces $\Sg_s:=\varphi_s(\Sg)$ are all $f$-stationary, and $\Om:=\varphi\big((-s_0,s_0)\times\Sg\big)$ is an open neighborhood of $\Sg$ in $M$ diffeomorphic to $(-s_0,s_0)\times\Sg$. 

Let us prove that the variation $\varphi$ does not increase the area. We use the subscript $s$ for denoting the quantities associated to $\Sg_s$. Define $(X_s)_p:=(\ptl\varphi/\ptl s)(s,p)$ and $u_s:=\escpr{X_s,N_s}$. Since $u_0=1$, we can suppose that $u_s>0$ on $\Sg_s$ for any $s\in (-s_0,s_0)$. As $\Sg_s$ is $f$-stationary, we  infer by Corollary~\ref{cor:stationary} (ii) that $\escpr{X_s,\nu_s}=0$ along $\ptl\Sg_s$. Thus equation~\eqref{eq:bd} yields $\ptl u_s/\ptl\nu_s+\text{II}(N_s,N_s)\,u_s=0$ along $\ptl\Sg_s$. For any $s\in (-s_0,s_0)$, let $H_f(s)$ be the constant $f$-mean curvature of $\Sg_s$. To prove the claim it suffices, by Lemma~\ref{lem:1st}, to show that $H'_f(s)\geq 0$ for any $s\in (-s_0,s_0)$. Note that $H_f'(s)=(\mathcal{L}_f)_s(u_s)$, where $(\mathcal{L}_f)_s$ is the $f$-Jacobi operator on $\Sg_s$ defined in \eqref{eq:hfprima}. Hence, we have
\begin{equation}
\label{eq:desi0}
H_f'(s)\int_{\Sg_s}\frac{1}{u_s}\,da_s=\int_{\Sg_s}\frac{(\mathcal{L}_f)_s(u_s)}{u_s}\,da_s.
\end{equation}
Let us see that the integrals at the right-hand side are nonnegative for any $s\in (-s_0,s_0)$. First note that
\begin{equation}
\label{eq:ausi}
-2\,\escpr{\nabla_{\Sg_s}\psi,\nabla_{\Sg_s} u_s}\leq |\nabla_{\Sg_s}\psi|^2\,u_s+\frac{|\nabla_{\Sg_s} u_s|^2}{u_s}.
\end{equation}
By taking into account \eqref{eq:hfprima}, \eqref{eq:deltaf} and \eqref{eq:yo1}, we get
\begin{align}
\label{eq:desi1}
\frac{(\mathcal{L}_f)_s(u_s)}{u_s}&=\frac{\Delta_{\Sg_s} u_s+\escpr{\nabla_{\Sg_s}\psi,\nabla_{\Sg_s} u_s}}{u_s}+\ric_f(N_s,N_s)+|\sg_s|^2
\\
\nonumber
&\geq\frac{\Delta_{\Sg_s} u_s}{u_s}-\frac{|\nabla_{\Sg_s}u_s|^2}{2\,u_s^2}+\frac{1}{2}\,\big(S_f+H_f(s)^2+|\sg_s|^2\big)-K_s+\Delta_{\Sg_s}\psi.
\end{align}
On the other hand, we can use the divergence theorem together with equalities \eqref{eq:yo2} and $\ptl u_s/\ptl\nu_s=-\text{II}(N_s,N_s)u_s$, to obtain
\begin{align*}
\int_{\Sg_s}\frac{\Delta_{\Sg_s} u_s}{u_s}\,da_s&=\int_{\Sg_s}\frac{|\nabla_{\Sg_s} u_s|^2}{u_s^2}\,da_s-\int_{\ptl\Sg_s}\frac{1}{u_s}\frac{\ptl u_s}{\ptl\nu_s}\,dl_s
\\
&=\int_{\Sg_s}\frac{|\nabla_{\Sg_s} u_s|^2}{u_s^2}\,da_s+\int_{\ptl\Sg_s} 2H_{\ptl M}\,dl_s-\int_{\ptl\Sg_s}h_s\,dl_s.
\end{align*}
By integrating and substituting the previous information into \eqref{eq:desi1}, we deduce
\begin{align}
\label{eq:desi2}
\int_{\Sg_s}\frac{(\mathcal{L}_f)_s(u_s)}{u_s}\,da_s&\geq\frac{1}{2}\,\int_{\Sg_s}\left(\frac{|\nabla_{\Sg_s} u_s|^2}{u_s^2}+S_f+H_f(s)^2+|\sg_s|^2\right)da_s-\int_{\Sg_s} K_s\,da_s
\\
\nonumber
&+\int_{\ptl\Sg_s}(H_f)_{\ptl M}\,dl_s-\int_{\ptl\Sg_s}h_s\,dl_s
\\
\nonumber
&=\frac{1}{2}\,\int_{\Sg_s}\left(\frac{|\nabla_{\Sg_s} u_s|^2}{u_s^2}+S_f+H_f(s)^2+|\sg_s|^2\right)da_s
\\
\nonumber
&+\int_{\ptl\Sg_s}(H_f)_{\ptl M}\,dl_s-2\pi\chi(\Sg_s),
\end{align}
where we have applied the Gauss-Bonnet theorem. Finally, the fact that $\Sg_s$ is topologically a cylinder together with hypotheses $S_f\geq 0$ and $(H_f)_{\ptl M}\geq 0$, allows us to conclude that $\int_{\Sg_s}(\mathcal{L}_f)_s(u_s)/u_s\,da_s\geq 0$, as we claimed. If we have equality for some $s\in(-s_0,s_0)$, then $\Sg_s$ is totally geodesic and the function $u_s$ is constant on $\Sg_s$. Moreover, by \eqref{eq:ausi} we get that the density $f$ is constant on $\Sg_s$.

Now, the first conclusion of the theorem follows as in Theorem~\ref{th:ricci} by using that $\Sg$ is locally weighted area-minimizing. In particular, $H_f(s)=0$ for any $s\in (-s_0,s_0)$, and so 
\[
0=H_f(s)=-\escpr{\nabla\psi,N_s}
\]
by equation \eqref{eq:fmc}. As a consequence, the density $f$ is constant in a neighborhood of $\Sg$. The second conclusion in the statement can also be deduced as in Theorem~\ref{th:ricci}.
\end{proof}

\begin{remark}
It is important to observe that our rigidity result in Theorem~\ref{th:ricci} does not follow from Theorem~\ref{th:main2} since the Perelman scalar curvature $S_f$ \emph{is not the trace} of $\ric_f$. This is a remarkable difference with respect to Riemannian geometry, where non-negative Ricci curvature implies non-negative scalar curvature. For example, in $\rr^3$ with the Gaussian density $f(p)=e^{-|p|^2}$, we have $\ric_f(v,v)=2|v|^2$ for any vector $v$, and $S_f(p)=12-4|p|^2$ for any $p\in\rr^3$. 
\end{remark}

\subsection{Perelman scalar curvature and area estimates}
\label{subsec:area}

The proof of Theorem~\ref{th:main} can be adapted to deduce upper and lower bounds for the weighted area of compact strongly stable surfaces. These area estimates involve a non-vanishing lower bound on the Perelman scalar curvature, and they are sharp, in the sense that the equality cases lead to rigidity results for area-minimizing surfaces. In this way we obtain results in the spirit of \cite[Thm.~3]{shen-zhu}, \cite{bbn}, \cite[Thm.~3 and Cor.~1]{nunes}, \cite[Thm.~1.2]{fbr}, \cite[Thms.~8 and 9]{ambrozio} for the Riemannian case, and of \cite[Sects.~8 and 9]{espinar} for surfaces with empty boundary in manifolds with density. The precise statements are the following. 

\begin{theorem}
\label{th:main3}
Let $M$ be a smooth oriented Riemannian $3$-manifold endowed with a density $f=e^\psi$ such that $\ptl M$ is $f$-mean convex and $S_f\geq S_0f$ on $M$ for some $S_0>0$. Consider a smooth, compact, connected, oriented, $f$-stationary surface $\Sg$ with $\emph{int}(\Sg)\sub\emph{int}(M)$ and non-empty boundary $\ptl\Sg\sub\ptl M$. If $\Sg$ is strongly $f$-stable, then $\Sg$ is topologically a disk with $A_f(\Sg)\leq 4\pi/S_0$. Moreover, if $\Sg$ is embedded, locally weighted area-minimizing, and $A_f(\Sg)=4\pi/S_0$, then $\Sg$ is a totally geodesic disk of constant Gauss curvature $(S_0f)/2$ bounded by geodesics, and there is an open neighborhood of $\Sg$ in $M$ which is isometric to a Riemannian product $(-\eps_0,\eps_0)\times\Sg$ with constant density. Finally, if $M$ is complete and $\Sg$ minimizes the weighted area in its isotopy class, then the density $f$ is constant in $M$ and the Riemannian product $\rr\times\Sg$ is an isometric covering of $M$.
\end{theorem}

\begin{theorem}
\label{th:main4}
Let $M$ be a smooth oriented Riemannian $3$-manifold endowed with a density $f=e^\psi$ such that $\ptl M$ is $f$-mean convex and $S_f\geq S_0f$ on $M$ for some $S_0<0$. Consider a smooth, compact, connected, oriented, $f$-stationary surface $\Sg$ with $\emph{int}(\Sg)\sub\emph{int}(M)$ and non-empty boundary $\ptl\Sg\sub\ptl M$. If $\Sg$ is strongly $f$-stable and has negative Euler characteristic $\chi$, then $A_f(\Sg)\geq 4\pi\chi/S_0$. Moreover, if $\Sg$ is embedded, locally weighted area-minimizing, and $A_f(\Sg)=4\pi\chi/S_0$, then $\Sg$ is a totally geodesic surface of constant Gauss curvature $(S_0f)/2$ bounded by geodesics, and there is an open neighborhood of $\Sg$ in $M$ which is isometric to a Riemannian product $(-\eps_0,\eps_0)\times\Sg$ with constant density. Finally, if $M$ is complete and $\Sg$ minimizes the weighted area in its isotopy class, then the density $f$ is constant in $M$ and the Riemannian product $\rr\times\Sg$ is an isometric covering of $M$.
\end{theorem}

\begin{remark}
Let $g$ be the genus of $\Sg$ and $m$ the number of boundary components. Then, the hypothesis $\chi<0$ is equivalent to that $g\geq 1$, or $g=0$ and $m\geq 3$. 
\end{remark}

\begin{proof}[Proof of Theorems~\ref{th:main3} and \ref{th:main4}]
The first part of the statements comes from the proof of Theorem~\ref{th:main}. In fact, equation \eqref{eq:yo3} yields
\[
0\leq\indo_f(u,u)\leq 2\pi\chi-\frac{1}{2}\int_\Sg S_f\,da\leq 2\pi\chi-\frac{S_0}{2}\,A_f(\Sg),
\]
from which we deduce the area estimates. Moreover, in the equality cases, we get that $f$ is constant on $\Sg$, the surface $\Sg$ is totally geodesic with constant Gauss curvature $(S_0f)/2$, $H_f=\ric_f(N,N)=0$ on $\Sg$, $\ptl\Sg$ consists of geodesics, and $\text{II}(N,N)=0$ along $\ptl\Sg$. In particular, we can use Proposition~\ref{prop:deformation} to construct a variation $\varphi:(-s_0,s_0)\times\Sg\to M$ of $\Sg$ such that $X=N$ on $\Sg$, all the hypersurfaces $\Sg_s$ are $f$-stationary, and $\Om:=\varphi\big((-s_0,s_0)\times\Sg\big)$ is an open neighborhood of $\Sg$ in $M$ diffeomorphic to $(-s_0,s_0)\times\Sg$. 

Now we show that $\varphi$ does not increase the weighted area. We use the same notation as in the proof of Theorem~\ref{th:main2}.  Let us see that $H'_f(s)\geq 0$ for any $s$ small enough, and that equality for some $s$ implies that $\Sg_s$ is totally geodesic and $u_s$ is constant on $\Sg_s$. By taking into account \eqref{eq:desi0} and \eqref{eq:desi2}, we obtain
\begin{align}
\label{eq:desi3}
H_f'(s)\,\int_{\Sg_s}\frac{1}{u_s}\,da_s&\geq\frac{1}{2}\int_{\Sg_s}\left(\frac{|\nabla_{\Sg_s}u_s|^2}{u_s^2}+S_f+H_f(s)^2+|\sg_s|^2\right)da_s
\\
\nonumber
&+\int_{\ptl\Sg_s}(H_f)_{\ptl M}\,dl_s-2\pi\chi(\Sg_s)
\\
\nonumber
&\geq\frac{1}{2}\int_{\Sg_s}S_f\,da_s+\int_{\ptl\Sg_s}(H_f)_{\ptl M}\,dl_s-2\pi\chi(\Sg_s)
\\
\nonumber
&\geq \frac{S_0}{2}\,A_f(s)-2\pi\chi,
\end{align}
where we have applied the hypotheses $S_f\geq S_0f$ and $(H_f)_{\ptl M}\geq 0$. If equality holds in \eqref{eq:desi3} for some $s\in (-s_0,s_0)$, then $\Sg_s$ is totally geodesic and $u_s$ is constant on $\Sg_s$. Moreover, by \eqref{eq:ausi} we also have that the density $f$ is constant on $\Sg_s$.

Suppose first that $S_0>0$ and $A_f(\Sg)=4\pi/S_0$. Since $\Sg$ is locally weighted area-minimizing, then $A_f(s)\geq A_f(0)$ in a small open interval $J$ containing $0$. Thus, equation \eqref{eq:desi3} gives us
\[
H_f'(s)\,\int_{\Sg_s}\frac{1}{u_s}\,da_s\geq\frac{S_0}{2}\,A_f(\Sg)-2\pi=0,
\]
for any $s\in J$, and the claim is proved.

Suppose now that $S_0<0$ and $A_f(\Sg)=4\pi\chi/S_0$. We reason as in the proof of \cite[Thm.~2]{micallef}. For any $s\in (0,s_0)$, we denote $\phi(s):=\int_{\Sg_s}(1/u_s)\,da_s$ and $\eta(s):=\int_{\Sg_s}u_s\,(da_f)_s$. From the first variation formula in Lemma~\ref{lem:1st}, equation \eqref{eq:desi3} reads
\begin{equation}
\label{eq:desi4}
H_f'(s)\geq\frac{S_0}{2\phi(s)}\,\int_0^sA_f'(t)\,dt=-\frac{S_0}{2\phi(s)}\int_0^sH_f(t)\,\eta(t)\,dt.
\end{equation}
Let $C>0$ be the maximum value of the function $\phi(s)^{-1}\int_0^s\eta(t)\,dt$ on $[0,s_0]$. Fix a number $\eps>0$ such that $-CS_0\,\eps<2$. If we show that $H_f(t)\geq 0$ for any $t\in [0,\eps)$, then we deduce from \eqref{eq:desi4} that $H_f'(s)\geq 0$ for any $s\in [0,\eps)$. This proves the claim by arguing similarly on some $(\eps',0]$. Suppose that there is $t_0\in (0,\eps)$ with $H_f(t_0)<0$. Let $t_*:=\inf B$, where $B:=\{t\in [0,t_0]\,;\,H_f(t)\leq H_f(t_0)\}$. Note that $t_*\leq t_0<\eps$. Let us assume $t_*>0$. By the definition of $t_*$ we have $H_f(t)\geq H_f(t_0)$ for any $t\in [0,t_*]$ and $H_f(t_0)=H_f(t_*)$. On the other hand, we apply the mean value theorem to find $t_1\in (0,t_*)$ such that $H_f'(t_1)=H_f(t_*)/t_*$. From \eqref{eq:desi4} we would obtain
\begin{align*}
\frac{H_f(t_*)}{t_*}=H_f'(t_1)&\geq-\frac{S_0}{2\phi(t_1)}\int_0^{t_1}H_f(t)\,\eta(t)\,dt\geq-\frac{S_0}{2}\,H_f(t_*)\,\frac{1}{\phi(t_1)}\int_0^{t_1}\eta(t)\,dt
\\
&\geq-\frac{CS_0}{2}\,H_f(t_*)>\frac{H_f(t_*)}{\eps},
\end{align*}
which is a contradiction since $t_*<\eps$. As a consequence, we would get $t_*=0$, and so $0=H_f(0)\leq H_f(t_0)<0$, which is another contradiction.

Finally, the conclusions of both theorems follow as in the proof of Theorem~\ref{th:main2}.
\end{proof}

\begin{remark}
Following classical terminology, we say that a strongly $f$-stationary hypersurface $\Sg$ is \emph{strictly $f$-stable} if the second derivative of the weighted area functional satisfies $A_f''(0)>0$ for any variation of $\Sg$. Clearly, a strictly $f$-stable hypersurface is locally weighted area-minimizing. In particular, all the results in this section hold for strictly $f$-stable hypersurfaces.
\end{remark}

\providecommand{\bysame}{\leavevmode\hbox to3em{\hrulefill}\thinspace}
\providecommand{\MR}{\relax\ifhmode\unskip\space\fi MR }
\providecommand{\MRhref}[2]{%
  \href{http://www.ams.org/mathscinet-getitem?mr=#1}{#2}
}
\providecommand{\href}[2]{#2}


\begin{thebibliography}{10}

\bibitem{ambrozio}
L.~C. Ambrozio, \emph{Rigidity of area-minimizing free boundary surfaces in
  mean convex three-manifolds}, arXiv:1301.6257, to appear in J.~Geom.~Anal.

\bibitem{be}
D.~Bakry and M.~{\'E}mery, \emph{Diffusions hypercontractives}, S\'eminaire de
  probabilit\'es, {XIX}, 1983/84, Lecture Notes in Math., vol. 1123, Springer,
  Berlin, 1985, pp.~177--206. \MR{MR889476 (88j:60131)}

\bibitem{bdce}
J.~L. Barbosa, M.~P. do~Carmo, and J.~Eschenburg, \emph{Stability of
  hypersurfaces of constant mean curvature in {R}iemannian manifolds}, Math. Z.
  \textbf{197} (1988), no.~1, 123--138. \MR{MR917854 (88m:53109)}

\bibitem{bayle-thesis}
V.~Bayle, \emph{Propri\'et\'es de concavit\'e du profil isop\'erim\'etrique et
  applications}, Ph.D. thesis, Institut Fourier (Grenoble), 2003.

\bibitem{bbn}
H.~Bray, S.~Brendle, and A.~Neves, \emph{Rigidity of area-minimizing
  two-spheres in three-manifolds}, Comm. Anal. Geom. \textbf{18} (2010), no.~4,
  821--830. \MR{2765731 (2012a:53067)}

\bibitem{cabre3}
X.~Cabr{{\'e}}, X.~Ros-Oton, and J.~Serra, \emph{{S}harp isoperimetric
  inequalities via the {ABP} method}, arXiv:1304.1724v3, April 2013.

\bibitem{cai-galloway}
M.~Cai and G.~J. Galloway, \emph{Rigidity of area minimizing tori in
  3-manifolds of nonnegative scalar curvature}, Comm. Anal. Geom. \textbf{8}
  (2000), no.~3, 565--573. \MR{1775139 (2001j:53051)}

\bibitem{cvm}
A.~Ca{\~n}ete, M.~Miranda, and D.~Vittone, \emph{Some isoperimetric problems in
  planes with density}, J. Geom. Anal. \textbf{20} (2010), no.~2, 243--290.
  \MR{2579510 (2011a:49102)}

\bibitem{homostable}
A.~Ca{\~n}ete and C.~Rosales, \emph{Compact stable hypersurfaces with free
  boundary in convex solid cones with homogeneous densities},
  arXiv:1304.1438v2, April 2013.

\bibitem{chang}
S.-Y.~A. Chang, M.~J. Gursky, and P.~Yang, \emph{Conformal invariants
  associated to a measure}, Proc. Natl. Acad. Sci. USA \textbf{103} (2006),
  no.~8, 2535--2540. \MR{2203156 (2007e:53032)}

\bibitem{fbr}
J.~Chen, A.~Fraser, and C.~Pang, \emph{Minimal immersions of compact bordered
  {R}iemann surfaces with free boundary}, arXiv:1209.1165, to appear in
  Trans.~Amer.~Math.~Soc.

\bibitem{mejia}
X.~Cheng, T.~Mejia, and D.~Zhou, \emph{{S}tability and compactness for complete
  $f$-minimal surfaces}, arXiv:1210.8076, to appear in Trans.~Amer.~Math.~Soc.

\bibitem{dhht}
A.~D\'iaz, N.~Harman, S.~Howe, and D.~Thompson, \emph{Isoperimetric problems in
  sectors with density}, Adv.~Geom. \textbf{12} (2012), 589--619.

\bibitem{calibrations}
T.~H. Doan, \emph{Some calibrated surfaces in manifolds with density}, J. Geom.
  Phys. \textbf{61} (2011), no.~8, 1625--1629. \MR{2802497 (2012e:53091)}

\bibitem{espinar}
J.~M. Espinar, \emph{{M}anifolds with density, applications and gradient
  {S}chr{\"o}dinger operators}, arXiv:1209.6162v6, November 2012.

\bibitem{fan}
E.~M. Fan, \emph{Topology of three-manifolds with positive {$P$}-scalar
  curvature}, Proc. Amer. Math. Soc. \textbf{136} (2008), no.~9, 3255--3261.
  \MR{2407091 (2009c:53050)}

\bibitem{fcs}
D.~Fischer-Colbrie and R.~Schoen, \emph{The structure of complete stable
  minimal surfaces in {$3$}-manifolds of nonnegative scalar curvature}, Comm.
  Pure Appl. Math. \textbf{33} (1980), no.~2, 199--211. \MR{MR562550
  (81i:53044)}

\bibitem{gromov-GAFA}
M.~Gromov, \emph{Isoperimetry of waists and concentration of maps}, Geom.
  Funct. Anal. \textbf{13} (2003), no.~1, 178--215. \MR{MR1978494
  (2004m:53073)}

\bibitem{ho}
P.~T. Ho, \emph{The structure of {$\phi$}-stable minimal hypersurfaces in
  manifolds of nonnegative {$P$}-scalar curvature}, Math. Ann. \textbf{348}
  (2010), no.~2, 319--332. \MR{2672304}

\bibitem{equations}
O.~A. Ladyzhenskaya and N.~N. Ural'tseva, \emph{Linear and quasilinear elliptic
  equations}, Translated from the Russian by Scripta Technica, Inc. Translation
  editor: Leon Ehrenpreis, Academic Press, New York, 1968. \MR{0244627 (39
  \#5941)}

\bibitem{lich1}
A.~Lichnerowicz, \emph{Vari{\'e}t{\'e}s riemanniennes {\`a} tenseur {C} non
  n{\'e}gatif}, C. R. Acad. Sci. Paris S{\'e}r. A-B \textbf{271} (1970),
  A650--A653. \MR{0268812 (42 \#3709)}

\bibitem{lich2}
\bysame, \emph{Vari{\'e}t{\'e}s k{\"a}hl{\'e}riennes {\`a} premi{\`e}re classe
  de {C}hern non negative et vari{\'e}t{\'e}s riemanniennes {\`a} courbure de
  {R}icci g{\'e}n{\'e}ralis{\'e}e non negative}, J. Differential Geom.
  \textbf{6} (1971/72), 47--94. \MR{0300228 (45 \#9274)}

\bibitem{liu}
G.~Liu, \emph{{S}table weighted minimal surfaces in manifolds with nonnegative
  {B}akry-{E}mery {R}icci tensor}, arXiv:1211.3770, to appear in
  Comm.~Anal.~Geom.

\bibitem{micallef}
M.~Micallef and V.~Moraru, \emph{{S}plitting of $3$-manifolds and rigidity of
  area-minimising surfaces}, arXiv:1107.5346, to appear in
  Proc.~Amer.~Math.~Soc.

\bibitem{gmt}
F.~Morgan, \emph{Geometric measure theory. {A} beginner's guide}, fourth ed.,
  Elsevier/Academic Press, Amsterdam, 2009. \MR{2455580 (2009i:49001)}

\bibitem{nardi}
G.~Nardi, \emph{{S}chauder estimate for solutions of {P}oisson's equation with
  {N}eumann boundary condition}, arXiv:1302.4103v2, April 2013.

\bibitem{nunes}
I.~Nunes, \emph{Rigidity of area-minimizing hyperbolic surfaces in
  three-manifolds}, J. Geom. Anal. \textbf{23} (2013), no.~3, 1290--1302.
  \MR{3078354}

\bibitem{perelman}
G.~Perelman, \emph{The entropy formula for the {R}icci flow and its geometric
  applications}, arXiv:math/0211159.

\bibitem{cones}
M.~Ritor{{\'e}} and C.~Rosales, \emph{Existence and characterization of regions
  minimizing perimeter under a volume constraint inside {E}uclidean cones},
  Trans. Amer. Math. Soc. \textbf{356} (2004), no.~11, 4601--4622. \MR{2067135
  (2005g:49076)}

\bibitem{ros-free}
A.~Ros, \emph{Stability of minimal and constant mean curvature surfaces with
  free boundary}, Mat. Contemp. \textbf{35} (2008), 221--240. \MR{2584186
  (2011b:53017)}

\bibitem{ros-souam}
A.~Ros and R.~Souam, \emph{On stability of capillary surfaces in a ball},
  Pacific J. Math. \textbf{178} (1997), no.~2, 345--361. \MR{1447419
  (98c:58029)}

\bibitem{ros-vergasta}
A.~Ros and E.~Vergasta, \emph{Stability for hypersurfaces of constant mean
  curvature with free boundary}, Geom. Dedicata \textbf{56} (1995), no.~1,
  19--33. \MR{1338315 (96h:53013)}

\bibitem{rosales-sr}
C.~Rosales, \emph{Complete stable {CMC} surfaces with empty singular set in
  {S}asakian sub-{R}iemannian $3$-manifolds}, Calc. Var. Partial Differential
  Equations \textbf{43} (2012), no.~3--4, 311--345.

\bibitem{rcbm}
C.~Rosales, A.~Ca{\~n}ete, V.~Bayle, and F.~Morgan, \emph{On the isoperimetric
  problem in {E}uclidean space with density}, Calc. Var. Partial Differential
  Equations \textbf{31} (2008), no.~1, 27--46. \MR{2342613 (2008m:49212)}

\bibitem{schoen-yau}
R.~Schoen and S.~T. Yau, \emph{Existence of incompressible minimal surfaces and
  the topology of three-dimensional manifolds with nonnegative scalar
  curvature}, Ann. of Math. (2) \textbf{110} (1979), no.~1, 127--142.
  \MR{541332 (81k:58029)}

\bibitem{shen-zhu}
Y.~Shen and S.~Zhu, \emph{Rigidity of stable minimal hypersurfaces}, Math. Ann.
  \textbf{309} (1997), no.~1, 107--116. \MR{1467649 (98g:53113)}

\bibitem{simon}
L.~Simon, \emph{Lectures on geometric measure theory}, Proceedings of the
  Centre for Mathematical Analysis, Australian National University, vol.~3,
  Australian National University Centre for Mathematical Analysis, Canberra,
  1983. \MR{756417 (87a:49001)}

\bibitem{simons-james}
J.~Simons, \emph{Minimal varieties in riemannian manifolds}, Ann. of Math. (2)
  \textbf{88} (1968), 62--105. \MR{0233295 (38 \#1617)}

\end{thebibliography}
\end{document}